\documentclass[12pt,a4paper]{article}

\usepackage[utf8]{inputenc}
\usepackage{amsmath}
\usepackage{amsfonts}
\usepackage{amssymb}
\usepackage{amsthm}
\usepackage[colorlinks]{hyperref}
\hypersetup{
  citecolor={cyan}
}
\usepackage[left=2cm,right=2cm,top=2cm,bottom=2cm]{geometry}
\usepackage{xcolor}
\usepackage[all,cmtip]{xy}
\usepackage{tikz}
\usetikzlibrary{decorations.markings}
\usepackage{enumerate}
\usepackage{array}

\newcolumntype{C}[1]{>{\centering\arraybackslash}m{#1}}

\newcommand{\id}{\mathrm{id}}

\newtheorem{thm}{Theorem}[section]
\newtheorem*{thm*}{Theorem}
\newtheorem{prop}[thm]{Proposition}
\newtheorem{cor}[thm]{Corollary}
\newtheorem*{cor*}{Corollary}
\newtheorem{lem}[thm]{Lemma}

\theoremstyle{definition}
\newtheorem{rem}[thm]{Remark}
\newtheorem{defn}[thm]{Definition}
\newtheorem{ex}[thm]{Example}

\DeclareMathOperator*{\colim}{colim}

\begin{document}
\title{Torus equivariant algebraic models and compact realization}
\author{Leopold Zoller}
\maketitle

\begin{abstract}
Let $T$ be a compact torus. We prove that, up to equivariant rational equivalence, the category of $T$-simply connected, $T$-finite type $T$-spaces with finitely many isotropy types is completely described by certain finite systems of commutative differential graded algebras with consistent choices of degree $2$ cohomology classes. We show that the algebraic systems corresponding to finite $T$-CW-complexes are exactly those which satisfy the necessary condition imposed by the Borel localization theorem along with certain finiteness conditions. We derive an algebraic characterization of when an algebra over a polyonmial ring is realized as the rational equivariant cohomology of a finite $T$-CW-complex. As further applications we prove that any GKM graph cohomology is realized by a finite $T$-CW-complex and classify equivariant cohomology algebras of finite $S^1$-CW-complexes with discrete fixed points.
\end{abstract}

\section{Introduction}

A core question in algebraic topology is whether a given ring is realizable as the cohomology ring of a space. While this is a widely open question for integral cohomology, it was proved in \cite{Quillen} that any nonnegatively graded, finite type, commutative algebra $A$ over $\mathbb{Q}$ with $A^0=\mathbb{Q}$, $A^1=0$ is realizable as $A\cong H^*(X;\mathbb{Q})$ for some simply connected space $X$. A main motivation for this article is to answer the question which algebras are realizable as the equivariant cohomology $H_T^*(X;\mathbb{Q})$ for some compact torus $T$. Naturally these algebras come equipped with a map $R\rightarrow H_T^*(X)$, where $R:=H^*(BT;\mathbb{Q})$ is the cohomology of the classifying space. Hence the question should be asked for $R$-algebras. It is not hard to deduce from nonequivariant realization that any finite type $R$-algebra $A$ is realizable as $H_T^*(X)$ for some free $T$-action on a simply connected space $X$ provided $A^1=0$ and $R^2\otimes A^0\rightarrow A^2$ is injective (see e.g.\ \cite[Proposition 4.2]{HalperinTori}). While at first glance this solves the problem, the answer is a little unsatisfactory: one can check that outside of the case where $\dim_\mathbb{Q}H_T^*(X)$ is finite --which it is usually not-- such a free realization $X$ has to be infinite dimensional. In fact, finding nice realizations has to be much harder, as for sufficiently nice spaces, there is powerful machinery linking algebraic properties of $H_T^*(X)$ to geometric properties of the action. One of the most prominent theorem in this direction is the Borel localization theorem which requires the $T$-space to be compact or satisfy certain other finiteness conditions (cf.\ \cite[Section 3.2]{AP}).
Consequently the interesting question to ask is: when is an $R$-algebra realized as the equivariant cohomology of a compact $T$-space or even a finite $T$-CW-complex? This turns out to be a more sophisticated problem as is illustrated e.g.\ by the fact that for $T=S^1$, $R\cong \mathbb{Q}[x]$, $c\in\mathbb{Q}$, the graded $\mathbb{Q}[x]$-algebra $\mathbb{Q}[x,a]/(a^2-cx^2)$ is realizable by a finite $S^1$-CW-complex if and only if $c$ is a square (see Example \ref{ex:nonrealizable}). Hence the above realizability question relies heavily on the multiplicative structure, to the point that the outcome might change when passing to $\mathbb{R}$-coefficients.

In the nonequivariant case, the realization question for cohomology algebras was solved by giving an algebraic description of the category of simply connected, finite $\mathbb{Q}$-type spaces up to rational equivalence. In the same vein, to attack this question we first give an algebraic description of the category of $T$-simply connected, $T$-finite $\mathbb{Q}$-type $T$-spaces up to equivariant rational equivalence. Here a $T$-space $X$ is called $T$-simply connected (resp.\ $T$-finite $\mathbb{Q}$-type) if every path component of $X^H$ is simply connected (resp.\ if $\dim H^k(X^H;\mathbb{Q})<\infty$ for $k\geq 0$) for all $H\subset T$. An equivariant rational equivalence is an equivariant map $X\rightarrow Y$ such that $H_*(X^H;\mathbb{Q})\rightarrow H_*(Y^H;\mathbb{Q})$ is an isomorphism for all $H\subset T$. Mostly for technical reasons, we restrict to spaces with finitely many orbit types (see Remark \ref{rem:disclaimer}).

Of course, this algebraization of the equivariant rational homotopy category is of interest independently of the cohomology realization problem. In fact, an algebraic description of this category has already been achieved in \cite{ScullMendes} using the following approach (in slightly greater generality): define $\mathcal{S}$ to be the set of pairs $(U,H)$ of subgroups of $T$, with $U\subset H$ and $U$ connected. Then consider the functor which maps a $T$-space $X$ to the diagram of cochain algebras, with entries $A_{PL}(X^H_{T/U})$ for any $(U,H)\in\mathcal{S}$, where $A_{PL}$ denotes the polynomial de Rham algebra and $X_{T/U}^H$ is the Borel construction of the $T/U$-action on the $H$-fixed points.
Formally this diagram is itself a functor $\mathcal{S}\rightarrow \mathrm{cdga}^{\geq 0}$ with values in the category of commutative cochain algebras, when regarding $\mathcal{S}$ as a category in a suitable fashion. We call this an $\mathcal{S}$-system. Denoting the $\mathcal{S}$-system of a $T$-space $X$ by $\underline{A_{PL}(X)}$, we obtain a map $\underline{A_{PL}(*)}\rightarrow \underline{A_{PL}(X)}$ induced by the constant map. At the $(U,H)$ position of $\underline{A_{PL}(*)}$  we have $A_{PL}(B(T/U))$. Thus pointwise, this captures the algebraic data of the respective Borel fibrations. Choosing a nice model $\underline{P}\rightarrow \underline {A_{PL}(*)}$ we obtain a functor from $T$-spaces to systems of cochain algebras under $\underline{P}$, i.e.\ morphisms of systems $\underline{P}\rightarrow \underline{A}$. Through this functor, the authors of \cite{ScullMendes} achieve an algebraic description of the $T$-equivariant rational homotopy category. However, as is stated in \cite{ScullMendes}, the object $\underline{P}$ is hard to describe explicitly.
The reason for this is that $\underline{P}\rightarrow \underline{A_{PL}(*)}$ has to respect the strict commutativity of the diagrams. Thus one can not freely choose Sullivan models for $B(T/U)$ pointwise, which would in general just produce homotopy commutative systems.

In this paper we seek to overcome this problem by showing that the datum of the $\underline{P}$-algebra structure is only relevant on the cohomological level and, in particular, the specific choice of $\underline{P}$ is irrelevant. This enables applications and lets us choose simpler algebraic models, which -- in the right setting -- can be written down explicitly. More precisely, we use the functor $X\rightarrow \underline{A_{PL}(X)}$ as above but at each entry $(U,H)$ of the diagram we only consider the additional datum $H^*(B(T/U))\rightarrow H^*(X^H_{T/U})$ on the cohomological level instead of the cochain level. We call this a cohomology $\underline{R}$-structure, where $\underline{R}=H^*(\underline{A_{PL}(*)})$ is the system of all cohomologies of the classifying spaces $B(T/U)$. We prove (see Theorem \ref{thm:R-realization})

\begin{thm*}
Let $C$ be a finite collection of subgroups of $T$ and $\mathcal{D}\subset \mathcal{S}$ be the stable subset generated by $C$. Let $\underline{A}$ be a $\mathcal{D}$-system with a cohomology $\underline{R}$-structure. Then there is a $T$-simply connected, $T$-finite $\mathbb{Q}$-type $T$-CW-complex $X$ with isotropies contained in $C$ such that $\underline{A_{PL}(X)}$ is connected to $\underline{A}$ via quasi isomorphisms of systems respecting the cohomology $\underline{R}$-structures if and only if $\underline{A}$ satisfies the triviality condition, is spacelike, of finite type, $H^1(\underline{A})=0$, and $\underline{R}^2\otimes H^0(\underline{A})\rightarrow H^2(\underline{A})$ is injective. The space $X$ can be chosen as a finite $T$-CW-complex if and only if $\underline{A}$ satisfies the localization condition and $H^*(\underline{A}(U,H))$ is finitely generated as an $R_{T/U}$-module for any $(U,H)\in\mathcal{D}$.
\end{thm*}

The localization condition in the second half of the theorem is just the necessary condition for realization by a sufficiently nice space, which is imposed by the Borel localization theorem. Thus this can be seen as a sort of ``converse to Borel localization''. This is the key to attacking the compact cohomology realization problem discussed in the beginning.

We point out that the first half of the above theorem can be obtained by drawing upon the results of \cite{ScullMendes} to obtain realizations for systems under $\underline{P}$ (as above) and, on the algebraic side, transition to cohomology $\underline{R}$-structures by the means given in this article. The key ingredient in this realization is Elmendorf's work from \cite{Elmendorf}. The statements on the occurring isotropy groups and realization through finite CW-complexes then become separate problems of approximating $T$-spaces up to rational homotopy within the equivariant world. However, in the present article we give a different, rather self contained proof, which we believe to be of independent value. While behaving less nicely with respect to functoriality, our realization procedure allows us close control over the realizations. The idea is that the equivariant approximation procedure hinted at above can be applied outside of the equivariant world to approximate a space as the Borel construction of a $T$-space. Thus starting with a system of cochain algebras, we take (nonequivariant) geometric realizations through the standard realization functors. Then we transfer this realization to the equivariant world by approximating the spaces as homotopy quotients of a $T$-CW-complex via an induction over the isotropy groups.

With regards to the compact realization problem of equivariant cohomology algebras, we arrive at the following criterion (Theorem \ref{thm:cohorealization}). We identify $H^*(BT)=\mathbb{Q}[V]$ for some vector space $V$ generated in degree $2$.

\begin{thm*}
Let $A$ be a graded $\mathbb{Q}[V]$-algebra. Then $A$ is isomorphic to the equivariant cohomology algebra of a $T$-simply connected finite $T$-CW-complex if and only if there is a finite collection $V_0,\ldots, V_k\subset V$ of vector spaces, $\mathbb{Q}[V_i]$-algebras $A_i$ for $i=0,\ldots,k$, and for each inclusion $V_i\supset V_j$ a map $f_{ij}\colon A_i\rightarrow \mathbb{Q}[V_i] \otimes_{\mathbb{Q}[V_j]} A_j$ of $\mathbb{Q}[V_i]$-algebras satisfying the following properties:
\begin{enumerate}[(i)]
\item  $V_0=V$, $V_k=0$, and for any $i,j$ we have $V_i+ V_j=V_l$ for some $0\leq l\leq k$.
\item $A_0=A$, each $A_i$ is spacelike, finitely generated as $\mathbb{Q}[V_i]$-module, $A_i^1=0$, and $V_i\otimes A^0\rightarrow A^2$ is injective for $0\leq i \leq k$.
\item We have $(\id_{\mathbb{Q}[V_i]}\otimes_{\mathbb{Q}[V_j]} f_{jl})\circ f_{ij}=f_{il}$. Furthermore for any subspace $W\subset V$ and the maximal $V_j\in\{V_0,\ldots,V_k\}$ with the property that $V_j\subset W$, the map $f_{0j}$ becomes an isomorphism when localized at the multiplicative subset $S(W)\subset \mathbb{Q}[V]$  generated by $V\backslash W$.
\end{enumerate}
\end{thm*}

We give two more specific applications. The first is a simplified classification of the equivariant cohomology algebras of finite $S^1$-CW-complexes with discrete fixed points (Corollary \ref{cor:discretealgebras}). Note that this can also be used to show nonrealizability of certain algebras (see Example \ref{ex:nonrealizable}). We identify $H^*(BS^1)\cong \mathbb{Q}[x]$.

\begin{cor*}
Let $A$ be a $\mathbb{Q}[x]$-algebra which is finitely generated as a $\mathbb{Q}[x]$-module, spacelike, $A^1=0$, and $\langle x\rangle_\mathbb{Q}\otimes A^0\rightarrow A^2$ is injective. Assume further that for $S=\{x^k~|~k\geq 0\}$, the localized graded algebra $S^{-1}A$ has no nilpotent elements. Then $A$ is isomorphic as a $\mathbb{Q}[x]$-algebra to the equivariant cohomology of a finite $S^1$-CW-complex if and only if $(S^{-1}A)^0\cong \mathbb{Q}\times\ldots\times \mathbb{Q}$ as $\mathbb{Q}$-algebras. In this case the realization can be chosen with discrete fixed point set and $S^1$-simply connected. Conversely any equivariant cohomology algebra of such a space is of the above type.
\end{cor*}

As a second application we show (Theorem \ref{thm:GKM-realization}, Remark \ref{ex:Egraph}) that the cohomology of an abstract GKM graph in the sense of \cite{GuilleminZara} is always realizable by a finite $T$-CW-complex whose one-skeleton is given by the graph. In fact, as we do not use those parts of the definition of an abstract GKM graph which model the combinatorial behaviour of $T$-manifolds, we work in a simplified setting, which we call $T$-graphs. 

\begin{thm*}
Let $(\Gamma,\alpha)$ be a $T$-graph. Then there is a finite $T$-CW-complex $X$ with $H^*_T(X)=H^*(\Gamma)$ and whose one-skeleton $X_1=\{x\in X~|~\mathrm{codim} T_x\leq 1\}$ is the graph of two-spheres defined by $(\Gamma,\alpha)$ if and only if for any codimension $1$ subtorus $H\subset T$, the graph $\Gamma^H$ is a disjoint union of trees. In this case $X$ can be chosen to be $T$-simply connected. In particular, any abstract GKM graph is realizable by a $T$-simply connected finite $T$-CW-complex.
\end{thm*}

Besides realization problems, there remains the question how much our algebraic models know about the $T$-spaces. It turns out that they contain all information needed to reconstruct the $T$-space up to equivariant rational equivalence. We obtain (cf. Corollary \ref{cor:categorialdescription})

\begin{thm*}
Let $C$ be a finite collection of subgroups and $\mathcal{D}\subset \mathcal{S}$ be the subset generated by $\mathcal{D}$.
The functor $X\mapsto \underline{A_{PL}(X)}$ embeds the homotopy category of $T$-simply connected, $T$-finite $\mathbb{Q}$-type $T$-spaces with isotropies in $C$ as a full subcategory of the homotopy category of $\mathcal{D}$-diagrams of cdgas with cohomology $\underline{R}$-structures.
\end{thm*}

For the proof we draw upon the corresponding statement for the algebraic models in \cite{ScullMendes}. The above theorem then comes down to showing that no information is lost when restricting to finite diagrams and keeping track of the Borel fibrations only on the cohomological level as opposed to the cochain level.

The article is structured as follows: Section 2 is an introduction to the main machinery and key definitions. As an example we discuss an explicit algebraic model describing a $T^2$-action on $S^6$. Section 3 deals with nice choices of models for the systems of algebras introduced in Section 2. It is the most technical of the sections and mainly serves as a compendium for the rest of the paper. In Section 4 we study the relation between cohomology $\underline{R}$-structures and morphisms $\underline{P}\rightarrow \underline{A}$ of (strictly commutative) systems whenever $\underline{P}$ is a nice model for the system $\underline{A_{PL}(*)}$ of classifying spaces. 
We postpone the discussion on implications on the homotopy categories until the final Section 7 as this is not needed for the rest of the paper. The goal of Section 5 is to prove the realization results of algebraic systems by $T$-spaces. Section 6 is devoted to applications of the realization result. Finally, as already mentioned, Section 7 serves to reformulate the results of Sections 4 and 5 in the language of homotopy categories.

We expect the reader to be familiar with the basics of nonequivariant rational homotopy theory and Sullivan models. We recommend \cite{BibelI} as an exposition. Furthermore, general familiarity with the concepts of equivariant cohomology and equivariant cell complexes is assumed, for which we point the reader towards \cite{AP}, \cite{tomDieck}. Knowledge of model categories (see e.g.\ \cite{Hirschhorn}) is assumed only in the last section. For any missing explanations of terminology, we refer to the above sources.

\noindent {\bf Acknowledgements.} The author wants to thank Oliver Goertsches an Panagiotis Konstantis for their comments on Section \ref{sec:GKM} and Michael Mandell for sharing his insight. The author is grateful to the Max Planck Institute for Mathematics in Bonn for its hospitality and financial support.

\section{The stratified cochains of a $T$-space}

Throughout the article, coefficients will be taken in the field $\mathbb{Q}$ (unless stated otherwise) and will be suppressed from the notation. This concerns cohomology as well as all occurring cdgas (commutative differential graded algebras). The latter are all assumed to be unital and non-negatively graded. They will also be referred to as (commutative) cochain algebras. Diagrams of cochain algebras, i.e.\ cochain algebra valued functors, will be written with underlines to distinguish them from cochain algebras.

\subsection{Basic definitions}

For any topological group $G$ let $EG\rightarrow BG$ denote the universal $G$-bundle. We fix the construction of universal bundles from \cite{Milnor}. It is functorial with respect to homomorphisms of groups. For a $G$-space $X$ define $X_G$ to be the quotient $(EG\times X)/G$ with respect to the diagonal action. The projection $X_G\mapsto BG$ onto the first component will be referred to as the Borel fibration. The construction $X\mapsto X_G$ is functorial in the following sense: consider a group homomorphism $\varphi\colon G\rightarrow H$ with induced map $\tilde{\varphi}\colon EG\rightarrow EH$, an $H$-space $Y$, and a map $f\colon X\rightarrow Y$ satisfying $f(g\cdot x)=\varphi(g)\cdot f(x)$ for any $g\in G$, $x\in X$. Then $\tilde{\varphi}\times f$ descends to a map $X_G\rightarrow Y_H$.
Note that the Borel fibration is induced by the constant map $X\mapsto *$ of $G$-spaces (with $\varphi=\id_G$).

Now let $T$ be a compact torus. We denote by $\mathcal{S}$ the category consisting of all pairs of subgroups $(U,H)$ of $T$ such that $U$ is connected and $U\subset H$. We turn this into a partially ordered set by writing $(U,H)\leq (U',H')$ if and only if $U\supset U'$ and $H\subset H'$. It will be convenient to regard $\mathcal{S}$ as a category, in which there is a unique morphism $(U,H)\rightarrow (U',H')$ whenever $(U,H)\leq (U',H')$.
For a fixed $T$-space $X$ and $(U,H)\leq (U',H')\in \mathcal{S}$, the projection $T/U\mapsto T/U'$ and the inclusion $X^{H'}\rightarrow X^H$ induce maps $X_{T/U'}^{H'}\rightarrow X_{T/U}^{H}$. Thus a $T$-space induces a contravariant functor from $\mathcal{S}$ into the category $\textbf{Top}$ of topological spaces. Clearly equivariant maps between these $T$-spaces induce natural transformations between the induced functors on $\mathcal{S}$.

\begin{rem}
The idea of regarding a $G$-space as a functor can be traced back to Elmendorf \cite{Elmendorf}. The specific functor \[\{T\text{-spaces}\}\longrightarrow \{ \text{Functors: }\mathcal{S}\rightarrow\textbf{Top}\}\] was introduced in \cite[Section 4]{ScullMendes}. We point out that we only need to consider $(U,H)\in\mathcal{S}$ with $U$ connected, due to the fact that we work over $\mathbb{Q}$.
\end{rem}

\begin{defn}
Any subset $\mathcal{D}\subset \mathcal{S}$ inherits the partial ordering and we regard it as a full subcategory of $\mathcal{S}$. A $\mathcal{D}$-system (of cochain algebras) is a covariant functor from $\mathcal{D}$ into the category of unital cochain algebras.
A morphism of $\mathcal{D}$-systems is a natural transformation between them. A quasi isomorphism is a morphism that is pointwise (i.e.\ for any fixed $(U,H)\in\mathcal{D}$) a quasi isomorphism.
\end{defn}

\begin{ex}
\begin{itemize}
\item
The fundamental bridge between geometry and algebra is the functor
\[\{T\text{-spaces}\}\longrightarrow \{\mathcal{S}\text{-systems}\},\qquad
X\mapsto\underline{A_{PL}(X)}\] where $\underline{A_{PL}(X)}(U,H)= A_{PL}(X^H_{T/U})$. If it is clear from the context that we work in a setting of $\mathcal{D}$-systems, we occasionally write $\underline{A_{PL}(X)}$ for the restriction $\underline{A_{PL}(X)}|_\mathcal{D}$.

\item We set $\underline{R}=H^*(\underline{A_{PL}(*)})$. Thus $\underline{R}(U,H)=H^*(B(T/U))=:R_{T/U}$. Any $T$-space $X$ comes with a morphism $\underline{R}\rightarrow H^*(\underline{A_{PL}(X)})$ defined by the Borel fibration. A $T$-equivariant map $X\rightarrow Y$ induces a commutative diagram \[\xymatrix{
H^*(\underline{A_{PL}(X)})& \ar[l] H^*(\underline{A_{PL}(Y)})\\
\underline{R}\ar[u]\ar[ur]&
}\]
\end{itemize}
\end{ex}

\begin{rem}
For $U\subset H$, the datum of the $T$-action on $X^H$ and the induced $T/U$-action on $X^H$ are essentially the same. On the algebraic side this is reflected by the fact that, using the ${R}_T$- and $R_{T/U}$-module structures, one can reconstruct algebras $A_{PL}(X^H_{T/U})$ and $A_{PL}(X^H_T)$ from each other up to quasi isomorphism. Thus the system $\underline{A_{PL}(X)}$ is more or less determined by its entries in the $(\{1\},H)$-positions. However, the value of considering diagrams with varying first component (instead of just fixing $T$ as the acting group) becomes apparent when formulating necessary conditions for systems to be realizable up to quasi isomorphism as $\underline{A_{PL}(X)}$. As an example, $A_{PL}(X^T_T)$ is up to quasi isomorphism of the form $R_T\otimes A_{PL}(X^T_{\{1\}})$ which is a nontrivial restriction (see the discussion on the triviality condition below).
\end{rem}

\begin{defn}
We call a subset $\mathcal{D}\subset \mathcal{S}$ stable if it is of the form $\{(U,H)\in \mathcal{D}_L\times \mathcal{D}_R~|~U\subset H\}$ where
\begin{itemize}
\item $\mathcal{D}_R$ is a collection of subgroups of $T$ which is stable under intersection and contains $\{1\},T$.
\item $\mathcal{D}_L$ is a collection of subtori of $T$ which is stable under intersection and for any $H\in \mathcal{D}_R$ the identity component $H_0$ is contained in $\mathcal{D}_L$.
\end{itemize}
\end{defn}

\begin{ex}
A collection $C$ of subgroups of $T$ generates a stable subset $\mathcal{D}(C)$ as follows: let $\mathcal{D}_R$ be the set containing all finite intersections of groups from $C\cup\left\{\{1\},T\right\}$ and let $\mathcal{D}_L$ be the set containing all identity components of groups from $\mathcal{D}_R$.
\end{ex}

\begin{defn}
We call a subset $\mathcal{D}\subset \mathcal{S}$ bounded if the length of strictly ascending chains within the partially ordered set $\mathcal{D}$ is bounded. We say a collection $C$ of subgroups of $T$ is bounded if $\mathcal{D}(C)$ is bounded.
\end{defn}

\begin{rem}\label{rem:disclaimer}
The subset given by all pairs of subtori is an example of a stable bounded subset. The definition essentially excludes certain scenarios where there occur groups with arbitrarily many connected components. For us it is important because it enables the inductive approach which is prevalent throughout the article. When approaching the main results we will occasionally restrict further and only consider finite systems. We do this since it conveniently simplifies some arguments as well as the notion of homotopy and aligns with the general spirit of the article. Generalizations beyond the finite case are oftentimes possible.
\end{rem}

For a stable subset $\mathcal{D}\subset \mathcal{S}$ and a subgroup $H\subset T$ we set $m_\mathcal{D}(H)\in \mathcal{D}_R$ to be the intersection over all subgroups in $\mathcal{D}_R$ which contain $H$. Note that for $(U,H)\in\mathcal{D}$ we have $(U,m_\mathcal{D}(H))\in\mathcal{D}$.
Now assume additionally that $\mathcal{D}_L$ consists of all subtori. Then for any $(U,H)\in\mathcal{S}$ the element $(U,m_\mathcal{D}(H))$ lies in $\mathcal{D}$. This defines a functor $\mathcal{S}\rightarrow\mathcal{D}$.
We denote by
\[r_\mathcal{D}\colon \{\mathcal{S}\text{-\textit{systems}}\}\rightarrow \{\mathcal{D}\text{-\textit{systems}}\}\quad \text{and}\quad i_\mathcal{D}\colon\{\mathcal{D}\text{-\textit{systems}}\}\rightarrow \{\mathcal{S}\text{-\textit{systems}}\}\]
the restriction and induction functors defined by $\mathcal{D}\subset \mathcal{S}$ and $m_\mathcal{D}\colon \mathcal{S}\rightarrow \mathcal{D}$. We have $r_\mathcal{D}\circ i_\mathcal{D}=\id$ and a natural transformation $\varphi_\mathcal{D}\colon \id\rightarrow i_\mathcal{D}\circ r_\mathcal{D}$ provided by $(U,H)\leq (U,m_\mathcal{D}(H))$.

\begin{lem}\label{lem:isotropytypes}
Let $X$ be a $T$-space, $C$ be the collection of all occurring isotropy groups of $X$ and $\mathcal{D}\subset \mathcal{S}$ a stable subset containing $\mathcal{D}(C)$. Then for any $K\subset T$ we have $X^K=X^{m_\mathcal{D}(K)}$. If additionally $\mathcal{D}_L$ consists of all subtori of $T$, then $(i_\mathcal{D}\circ r_\mathcal{D})(\underline{A_{PL}(X)})=\underline{A_{PL}(X)}$.
\end{lem}

\begin{proof}
Define $L$ as the intersection of all isotropy groups of $X$ which contain $K$. We have $K\subset m_\mathcal{D}(K)\subset L$. Thus $X^K=X^L$ implies $X^K=X^{m_\mathcal{D}(K)}$.
\end{proof}

\begin{ex}\label{ex:S6}
Consider $T=T^2$ and the $T$-space $X=S^6\subset \mathbb{C}^3\oplus \mathbb{R}$ with the action $(s,t)\cdot (v,w,z,h)=(sv,tw,st^{-1}z,h)$. The occurring isotropy groups are \[C=\{T,S^1\times 1,\Delta S^1,1\times S^1,\{1\}\}\],
where $\Delta S^1$ is the diagonal circle.
In view of the above discussion (and the main results of the paper), the equivariant rational homotopy information on $X$ is contained in the $\mathcal{D}(C)$-system $\underline{A_{PL}(X)}|_{\mathcal{D}(C)}$. We give an explicit description of a model $\underline{\mathcal{M}}\rightarrow \underline{A_{PL}(X)}|_{\mathcal{D}(C)}$. As a diagram of algebras $\underline{\mathcal{M}}$ is given by
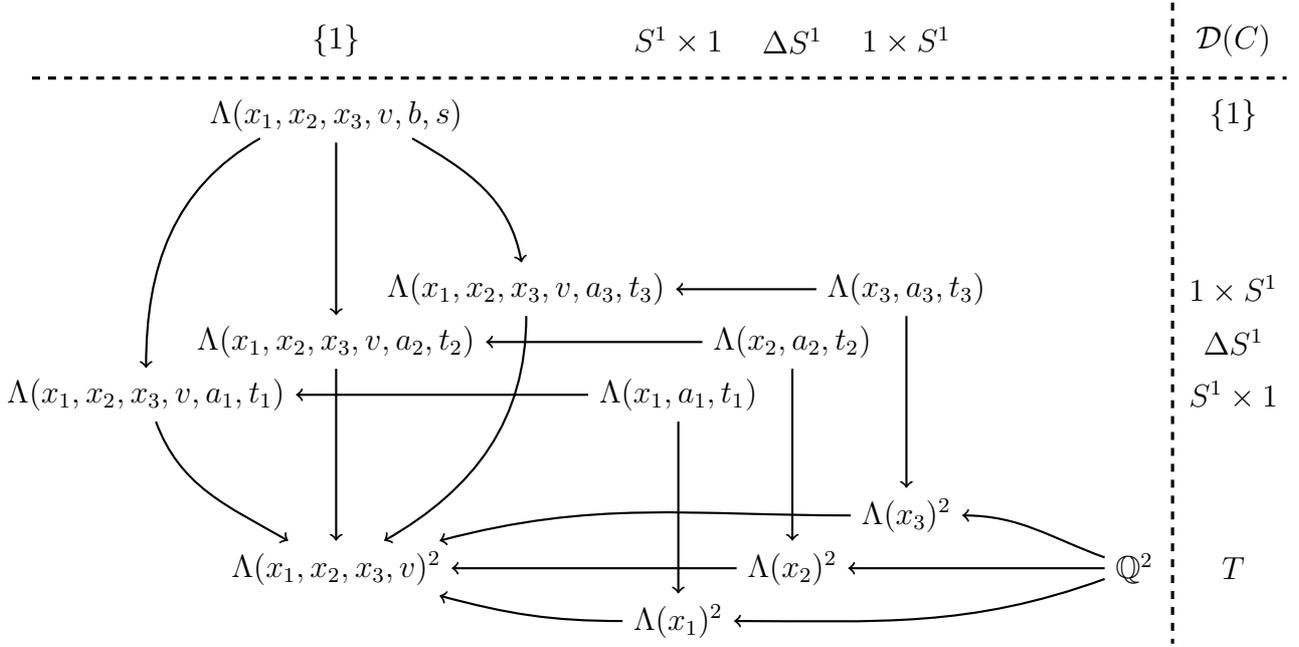
\begin{center}
\begin{tikzpicture}

\node at (0,0) (a) {$\Lambda (x_1,x_2,x_3,v)^2$};

\node at (-2.5,2.3) (b) {$\Lambda (x_1,x_2,x_3,v,a_1,t_1)$};
\node at (0,3) (c) {$\Lambda (x_1,x_2,x_3,v,a_2,t_2)$};
\node at (2.5,3.7) (d) {$\Lambda (x_1,x_2,x_3,v,a_3,t_3)$};

\node at (0,6) (e) {$\Lambda (x_1,x_2,x_3,v,b,s)$};

\node at (4.5,2.3) (f) {$\Lambda (x_1,a_1,t_1)$};
\node at (6,3) (g) {$\Lambda (x_2,a_2,t_2)$};
\node at (7.5,3.7) (h) {$\Lambda (x_3,a_3,t_3)$};

\node at (4.5,-0.7) (i) {$\Lambda (x_1)^2$};
\node at (6,0) (j) {$\Lambda (x_2)^2$};
\node at (7.5,0.7) (k) {$\Lambda (x_3)^2$};

\node at (10.5,0) (l) {$\mathbb{Q}^2$};

\draw[->, thick] (b) to[out=-70, in=150] (a);
\draw[->, thick] (c) to (a);
\draw[->, thick] (d) to[out=-90, in=30] (a);

\draw[->, thick] (e)+(-1,-0.3) to[out=-150, in=90] (b);
\draw[->, thick] (e) to (c);
\draw[->, thick] (e)+(1,-0.3) to[out=-30, in=100] (d);

\draw[->, thick] (f) to (b);
\draw[->, thick] (g) to (c);
\draw[->, thick] (h) to (d);

\draw[->, thick] (f) to (i);
\draw[->, thick] (g) to (j);
\draw[->, thick] (h) to (k);

\draw[->, thick] (i) to[out=180, in=-15] (a);
\draw[->, thick] (j) to[out=180, in=0] (a);
\draw[->, thick] (k) to[out=180, in=+15] (a);

\draw[->, thick] (l) to[out=160, in=0] (k);
\draw[->, thick] (l) to[out=180, in=0] (j);
\draw[->, thick] (l) to[out=200, in=0] (i);

\draw[very thick, dashed] (-4,6.5)--(12.5,6.5);
\draw[very thick, dashed] (11,7.5)--(11,-1);

\node at (0,7) {$\{1\}$};
\node at (4.5,7) {$S^1\times 1$};
\node at (6,7) {$\Delta S^1$};
\node at (7.5,7) {$1\times S^1$};
\node at (11.8,7) {$\mathcal{D}(C)$};
\node at (11.8,6) {$\{1\}$};
\node at (11.8,3.7) {$1\times S^1$};
\node at (11.8,3) {$\Delta S^1$};
\node at (11.8,2.3) {$S^1\times 1$};
\node at (11.8,0) {$T$};

\end{tikzpicture}
\end{center}
where the groups at the top and right indicate which entry of the diagram the respective algebra belongs to (e.g.\ $\Lambda(x_1,x_2,x_3,v)^2=\underline{\mathcal{M}}(\{1\},T)$, where the exponent in the first algebra denotes the Cartesian product). The degrees are $|v|=1$, $|x_i|=2=|a_i|$, $|b|=6$, $|t_i|=3$, and $|s|=11$. The differentials are $dv=x_1-x_2-x_3$, $dx_i=0=d a_i=db$, $dt_i=a_i^2-x_ia_i$, and $ds=b^2-x_1x_2x_3b$. The horizontal maps are the obvious inclusions. For the vertical maps in the middle row we specify $x_i\mapsto (x_i,x_i)$, $v\mapsto(v,v)$, $a_i\mapsto(0,x_i)$, and $t_i\mapsto 0$. Finally the three arrows leaving $\underline{\mathcal{M}}(\{1\},\{1\})$ are defined by $b\mapsto x_2x_3a_1$, $x_1x_3a_2$, $x_1x_2a_3$ and $s\mapsto x_2^2x_3^2t_1$, $x_1^2x_3^2t_2$, $x_1^2x_2^2t_3$ respectively.

Let us briefly verify that this is in fact a model for $\underline{A_{PL}(X)}$. We focus on the diagram aspects, leaving some pointwise verifications to the reader.

We begin in the bottom row. The space $X^T$ consists of two discrete points so $\mathbb{Q}^2$ is a model for $A_{PL}(X^T)$. Denote the groups $S^1\times 1$, $\Delta S^1$, and $1\times S^1$ by $H_1$, $H_2$, $H_3$ respectively. Each quotient $T/H_i$ is a circle and we choose minimal models $\Lambda(x_i)\rightarrow A_{PL}(B(T/H_i))$. Taking the two-fold product then gives the desired maps $\mathcal{M}(H_i,T)\rightarrow A_{PL}(X^T_{T/H_i})$. If we scale the $x_i$ correctly, then their images $[x_1],[x_2],[x_3]$ in $H^*(BT)$ after composing with the respective maps $H^*(B(T/H_i))\rightarrow H^*(BT)$ can be checked to satisfy $[x_1]=[x_2]+[x_3]$. Thus we can choose an image of $v$ in $A_{PL}^1(BT)$ that is compatible with the differentials and yields the Sullivan model $\Lambda(x_1,x_2,x_3,v)\rightarrow A_{PL}(BT)$ extending the previous maps. This finishes considerations for the bottom row.

In the middle row we have $X^{H_i}=S^2$ for each of the $H_i$ with the $T/H_i$-action corresponding to standard rotation. We leave it to the reader to verify the existence of the dashed quasi isomorphism in a commutative diagram

\[\xymatrix{
\Lambda(x_i,a_i,t_i)\ar@{-->}[r]\ar[d] & A_{PL}(X_{T/H_i}^{H_i})\ar[d]\\
\Lambda(x_i)\times \Lambda(x_i)\ar[r] & A_{PL}(X^T_{T/H_i})
}
\]
where the image of $x_i$ is given by $\Lambda(x_i)\rightarrow A_{PL}(B(T/H_i))\rightarrow A_{PL}(X^{H_i}_{T/H_i})$ and all solid arrows are the previously described maps. We just note that strict commutativity of the square can be achieved using the surjectivity of the right hand vertical map. To obtain compatible extensions to maps $\underline{\mathcal{M}}(\{1\},H_i)\rightarrow A_{PL}(X^{H_i}_T)$, we use the compositions $\Lambda(x_i,a_i,t_i)\rightarrow A_{PL}(X^{H_i}_{T/H_i})\rightarrow A_{PL}(X^{H_i}_T)$ and $\Lambda(x_1,x_2,x_3,v)\rightarrow A_{PL}(BT)\rightarrow A_{PL}(X^{H_i}_T)$, which in fact agree on $x_i$ and thus piece together to a quasi isomorphism.

Finally, we sketch how to construct the remaining map $\underline{\mathcal{M}}(\{1\},\{1\})\rightarrow A_{PL}(X_T)$ such that it is compatible with all other maps. We write $\mathcal{D}(\{1\},\{1\})$ the subset of $\mathcal{D}(C)$ consisting of the $(\{1\},H_i)$ and $(\{1\},T)$. We have given maps $\underline{\mathcal{M}}(\{1\},\{1\})\rightarrow \lim_{\mathcal{D}(\{1\},\{1\})} \underline{\mathcal{M}}\rightarrow \lim_{\mathcal{D}(\{1\},\{1\})} \underline{A_{PL}(X)}$. By standard theory (cf.\ Lemma \ref{lem:subcomplexes}) there is a surjective quasi isomorphism $A_{PL}(Y_T)\rightarrow \lim_{\mathcal{D}(\{1\},\{1\})} \underline{A_{PL}(X)}$, where $Y_T$ is the union over all subcomplexes in the limit. Here $Y\subset X$ is just the one-skeleton of the action, consisting of three copies of $S^2$, joined at the poles. We can now lift to a map $\underline{\mathcal{M}}(\{1\},\{1\})\rightarrow A_{PL}(Y_T)$ and we are done if we can lift this through the map $A_{PL}(X_T)\rightarrow A_{PL}(Y_T)$. On the $x_i$ and $v$ this lift is just given by $\Lambda(x_1,x_2,x_3,v)\rightarrow A_{PL}(BT)\rightarrow A_{PL}(X_T)$. On $b,s$ this lift can be constructed by hand, using that $A_{PL}(X_T)\rightarrow A_{PL}(Y_T)$ is surjective and an isomorphism on cohomology in higher degrees.
\end{ex}

\subsection{Cohomology $\underline{R}$-structures}

In Example \ref{ex:S6}, at each point $(U,H)\in \mathcal{D}(C)$ of the diagram, there is a map $R_{T/U}\rightarrow \underline{\mathcal{M}}(U,H)$, which is a model for the Borel fibration. To obtain an explicit description, e.g.\ identify $R_{T/H_i}\cong \Lambda(x_i)$, $R_T\cong \Lambda(x_1,x_2,x_3)/(x_1-x_2-x_3)\cong\Lambda(x_1,x_2)$ and take the obvious inclusions (diagonal in the bottom row). However for the whole system this does not piece together to a morphism $\underline{R}\rightarrow \underline{\mathcal{M}}$. The reason is that w.r.t.\ these identifications the map $R_{T/H_3}\rightarrow R_T$ is defined by $x_3\mapsto x_1-x_2$, while in $\underline{\mathcal{M}}$ the equality $x_3=x_1-x_2$ does not hold on the cochain level. This causes problems with respect to strict commutativity. Note however that there are no problems to define $\underline{R}\rightarrow H^*(\underline{\mathcal{M}})$ on the level of cohomology, leading us to the following

\begin{defn}
A cohomology $\underline{R}$-structure on a $\mathcal{D}$-system $\underline{A}$ is a morphism $\underline{R}|_{\mathcal{D}}\rightarrow H^*(\underline{A})$.
\end{defn}

\begin{defn}
Let $\underline{A}$ be a $\mathcal{D}$-system with a cohomology $\underline{R}$-structure. We say that $\underline{A}$ satisfies the triviality condition \textbf{(TC)} if the map \[R_{T/K}\otimes_{R_{T/{U}}} H^*(\underline{A}(U,H))\rightarrow H^*(\underline{A}(K,H))\] is an isomorphism for every $(U,H),(K,H)\in \mathcal{D}$ with $K\subset U$.
\end{defn}

\begin{rem}\label{rem:TC}
This condition reflects the fact that the Borel fibration of a trivial action is just a product: let $X$ be a $T$-space and $U\subset T$ a subtorus which acts trivially. Then let $L\subset T$ be a complementary torus to $U$, i.e.\ $T=U\times L$ and $L\cong T/U$. Then $ET\simeq EU\times EL$, $X_T\simeq BU\times X_L$, and $R_T\cong R_U\otimes R_L\cong R_U\otimes R_{T/U}$. We obtain \[H_T^*(X)=H^*(BU)\otimes H^*_L(X)= R_U\otimes H_{T/U}^*(X)=R_T\otimes_{R_{T/U}} H_{T/U}^*(X).\]
In particular the canonical cohomology $\underline{R}$-structure of $\underline{A_{PL}(X)}$ satisfies \textbf{(TC)}.
\end{rem}

\begin{defn}
For any $(U,H)\in\mathcal{S}$ let $S_{T/U}(H)\subset R_{T/U}$ be the multiplicative subset generated by $R_{T/U}^2\backslash V$, with $V=\ker (H^2(B(T/U))\rightarrow H^2(B(H/U)))$.
Let $\underline{A}$ be a $\mathcal{D}$-system with a cohomology $\underline{R}$-structure. We say that $\underline{A}$ satisfies the localization condition \textbf{(LC)} if for any $(U,H)\in\mathcal{D}$ and any torus $K\supset H$, the localized map
\[S_{T/U}(K)^{-1}H^*(\underline{A}(U,H))\rightarrow S_{T/U}(K)^{-1} H^*(\underline{A}((U,m_\mathcal{D}(K)))\]
is an isomorphism.
\end{defn}

\begin{rem}\label{rem:localization}
\begin{enumerate}[(i)]
\item Note that $m_\mathcal{S}(H)=H$ for any $H\subset T$. Thus over $\mathbb{Q}$, the Borel localization theorem can be rephrased as follows: for a {suitably nice} $T$-space $X$ the $\mathcal{S}$-system $\underline{A_{PL}(X)}$ together with $\underline{R}\rightarrow H^*(\underline{A_{PL}(X)})$ satisfies \textbf{(LC)}. For the precise meaning of suitably nice we refer to \cite[Theorem 3.2.6]{AP}. We note however that it includes the case of a finite $T$-CW-complex.

\item Let $X$ be a finite $T$-CW-complex, $C$ be the collection of its isotropy groups, and $\mathcal{D}\subset \mathcal{S}$ be a stable subset which contains $\mathcal{D}(C)$. It follows Lemma \ref{lem:isotropytypes} that the restricted system $\underline{A_{PL}(X)}|_{\mathcal{D}}$ satisfies \textbf{(LC)}.
\end{enumerate}
\end{rem}

The following lemma is a crucial ingredient for the connection between \textbf{(LC)} and compact realization of systems.

\begin{lem}\label{lem:localizationfinite}
Let $f\colon M\rightarrow N$ be a map between finitely generated graded $R_{T/U}$-modules and assume that it induces an isomorphism when localized at $S_{T/U}(K)$ for any torus $K\supsetneq U$. Then the kernel and cokernel of $f$ are finite dimensional over $\mathbb{Q}$.
\end{lem}

\begin{proof}
Let $\mathfrak{p}\subset R_{T/U}$ be a prime ideal with $\mathfrak{p}\cap R^2_{T/U}\subsetneq R^2_{T/U}$.  Then there is a subtorus $K\supsetneq U$ such that the kernel of $H^*(B(T/U))\rightarrow H^*(B(K/U))$ is equal to $\mathfrak{p}\cap R^2_{T/U}$. Then $S_{T/U}(K)\subset R_{T/U}\backslash \mathfrak{p}$. It follows from the assumptions that the kernel $A$ of $f$ vanishes when localized at $\mathfrak{p}$.

Now in case $A\neq 0$, let $\mathfrak{p}\subset R_{T/U}$ be a prime ideal which contains the annihilator $\mathrm{ann}(A)$.
Note that $A$ is finitely generated since $M$ is and $R_{T/U}$ is Noetherian. It follows that $A_\mathfrak{p}\neq 0$.
By what we have seen, above this implies $\mathfrak{p}\cap R^2_{T/U}=R^2_{T/U}$ and thus $\mathfrak{p}=R_{T/U}^+$ is the maximal homogeneous ideal. Hence the radical is given by $\sqrt{\mathrm{ann(A)}}=R^+_{T/U}$. As a vector space, $A$ is generated over $\mathbb{Q}$ by elements $m y_i$ where the $m\in R_{T/H}^+$ is a monomial and the $y_j$ are a finite $R_{T/H}$-generating set of $A$. By what we have just seen, these expressions vanish when the degree of $m$ passes a certain point. It follows that $A$ is finite dimensional and the same argument works for the cokernel of $f$.
\end{proof}

\section{Models for systems of cochain algebras}

The goal of this section is to lay the technical founadtion for the rest of the paper by studying nice representatives of the quasi isomorphism types of the systems introduced in the previous section. The reader may notice relations to the notions of fibration and cofibration in the sense of model categories. In fact, conditions related to the surjectivity condition from Section \ref{sec:surjcond} appear in similar context in \cite{ScullS1}, \cite{Triantifillou} and have been interpreted in model categorial terms in \cite{ScullCats} in the setting of finite group actions. We will however not delve into the details of this viewpoint in this paper. We just point out that the properties studied in this section do not correspond to the projective model structure which occurs briefly as a tool in Section \ref{sec:homcats}: in the projective model structure every system is fibrant and being a system of Sullivan cdgas is not sufficient to be cofibrant.

\subsection{The disconnected Sullivan condition}\label{sec:disconnected}

\begin{defn}
We say a cochain algebra $A$ is
\begin{itemize}
\item \emph{connected} if $A^0=\mathbb{Q}$. It is cohomologically connected if $H^*(A)$ is connected.
\item \emph{simply connected} if $A^1=0$. It is cohomologically simply connected if $H^*(A)$ is simply connected.
\item \emph{spacelike} if $H^0(A)=\mathbb{Q}\times\ldots\times \mathbb{Q}$ as algebras.
\end{itemize}
A system of cochain algebras is said to be (cohomologically) connected, simply connected, or spacelike, if it satisfies the respective conditions pointwise, i.e.\ at any fixed position $(U,H)\in\mathcal{D}$.
\end{defn}

We recall some terminology which we try to keep consistent with \cite{BibelI}, \cite{BibelII}.

\begin{itemize}
\item A morphism of the form $(B,d)\rightarrow (B\otimes \Lambda Z,d)$, $b\mapsto b\otimes 1$ with $Z=Z^{\geq 0}$ will be called a free extension.

\item We say that a free extension satisfies the nilpotence condition (also called Sullivan condition), if $B$ is cohomologically connected and $Z=\bigcup_{r\geq 0} Z(r)$ is a union of graded subspaces with $d\colon Z(0)\rightarrow B$ and $d\colon Z(r)\rightarrow B\otimes \Lambda V(r-1)$ for $r\geq 1$. We also just call this a free nilpotent extension.

\item A relative Sullivan algebra is a free nilpotent extension as above such that $Z=Z^{\geq 1}$. If additionally $d\colon 1\otimes Z\rightarrow (\mathbb{Q}\oplus B^{\geq 1})\otimes \Lambda Z$ then we call it a Sullivan extension (the latter notion was introduced in \cite{BibelII} and is in practice not more restrictive than the first one cf.\ \cite[Lemma 14.8]{BibelI}).

\item A free nilpotent algebra (resp.\ a Sullivan algebra) is a free nilpotent (resp.\ Sullivan) extension of $\mathbb{Q}$. An almost Sullivan algebra is a product $(\Lambda Z,d)\otimes (\Lambda (U\oplus dU),d)$, where $(\Lambda Z,d)$ is a Sullivan algebra and $(\Lambda(U\oplus dU),d)$ is contractible, generated in degree $0$.

\item We say a morphism of $\mathcal{D}$-systems (resp.\ a $\mathcal{D}$-system) is a free (nilpotent)/Sullivan extension (resp.\ free (nilpotent)/(almost) Sullivan) if the respective condition is fulfilled pointwise.
\end{itemize}

Rather frequently we will deal with free nilpotent cdgas $(\Lambda Z,d)$ which have $Z^0\neq 0$ and are thus not quite Sullivan algebras (but usually almost Sullivan algebras). A reason for this is that oftentimes bigger models are needed in order to achieve strict commutativity of diagrams (an example being Proposition \ref{prop:Sullivandiagram} below). In \cite{BibelII} there is the notion of a $\Lambda$-cdga which allows generators in degree $0$. Unfortunately it has a built in minimality condition which we usually neither have nor need and is thus also not our notion of choice. We want to point out that the theory for relative Sullivan cdgas in \cite{BibelI} and $\Lambda$-extensions in \cite{BibelII} is developed rather analogously and many proofs carry over verbatim to the less restrictive case of free nilpotent extensions (usually the essential ingredient is the ``lifting lemma'', which holds equally in the general free nilpotent case). A summary of the technical statements we will need in this paper is given in Remark \ref{rem:technicalstuff}.

While many standard references focus mainly on cohomologically connected cochain algebras, the theory extends rather seamlessly to spacelike ones. The reason for this is the following Lemma (see e.g.\ \cite[Theorem B, Principle 3.2]{LazarevMarkl} for proofs and a deeper treatment of disconnected rational homotopy theory) which shows that the category of spacelike cochain algebras is completely described in terms of products in the cohomologically connected category.

\begin{lem}\label{lem:disconnected}
Let $A$ be a spacelike cochain algebra. There is a finite product decomposition $A=\prod_i A_i$ with $A_i$ cohomologically connected and which is unique up to permutation of the factors. If $B=\prod B_j$ is another spacelike cochain algebra with cohomologically connected $B_j$, then there is a bijection between homomorphisms $\varphi\colon A\rightarrow B$ and matrices of homomorphisms $\varphi_{ij}\colon A_i\rightarrow B_j$ such that for each $j$ only a single $\varphi_{ij}$ is nonzero.
\end{lem}

The product decomposition in the Lemma is also referred to as the decomposition into the path components of a spacelike cochain algebra.

\begin{proof}
Since $\mathbb{Q}^r\cong H^0(A,d)=\ker d|_{A_{0}}$ we obtain a unique collection of idempotent cocycles $1_1,\ldots,1_r$ such that $1=1_1+\ldots+1_r$. The factors in the product decomposition $A=\prod A_i$ are given by $A_i=1_i\cdot A$. The remaining properties are easily verified.
\end{proof}

\begin{defn}
A disconnected free extension of $B$ is a morphism $(B,d)\rightarrow \prod^k_{i=1}(B\otimes \Lambda Z_i,d)$ where each $B\rightarrow B\otimes \Lambda Z_i$ is a free extension. We say that it satisfies the nilpotence condition (resp.\ is a disconnected Sullivan extension) if each of the $B\rightarrow B\otimes \Lambda Z_i$ satisfies the nilpotence condition (resp.\ is a Sullivan extension). 
\end{defn}

\begin{prop}\label{prop:Sullivandiagram}
Let $\underline{P},\underline{A}$ be $\mathcal{D}$-systems of cochain algebras. Assume that $\underline{P}$ is cohomologically connected and $\underline{A}$ is spacelike. Then for any morphism $\underline{P}\rightarrow\underline{A}$ such that $H^1(\underline{P})\rightarrow H^1(\underline{A})$ is injective, there is a natural factorization $\underline{P}\rightarrow \underline{M}\rightarrow\underline{A}$ where the first map is a disconnected free nilpotent extension and the second map is a quasi isomorphism.
\end{prop}

\begin{proof}
This is just a consequence of the functorial construction of free nilpotent extensions given e.g.\ in \cite[Section 4.7]{BG}. We briefly recall the construction (in a slightly adapted version) in order to demonstrate that it has the claimed properties. Consider first a morphism $(P,d)\rightarrow (A,d_A)$ between cohomologically connected cochain algebras such that $H^1(P)\rightarrow H^1(A)$ is injective. Then we set $\mathcal{M}_0=P\otimes \Lambda V_0$ where $V_0\cong \ker (d_A)^{\geq 1}$. We set the differential of $\mathcal{M}_0$ to be trivial on $V_0$ and define $\psi_0\colon \mathcal{M}_0\rightarrow A$ in the obvious way. Thus $\psi_0$ is surjective on cohomology and an isomorphism on $H^0$.

Now for $k\geq 1$ and each pair $(x,y)\in A^{k-1}\times \mathcal{M}_0^k$ with $dy=0$ and $dx=\psi_0(y)$, introduce a generator $v_{(x,y)}$ in degree $k-1$ with $dv_{(x,y)}=y$ and $\psi_1(v_{(x,y)})=x$. We obtain a commutative diagram

\[\xymatrix{
\mathcal{M}_0\ar[r]\ar[d]_{\psi_0}& \mathcal{M}_0\otimes \Lambda V_1\ar[dl]^{\psi_1}\\ A &
}\]
where $V_1$ is defined as follows:
in degrees $\geq 1$, a basis of $V$ is given by the $v_{(x,y)}$ for all possible choices of pairs $(x,y)$ as above, while in degree $0$ we choose a collection $(x_i,y_i)$ such that the $y_i$ descend to a basis of $\ker(\psi_0^*\colon H^1(\mathcal{M}_0)\rightarrow H^1(A))$. We set $\mathcal{M}_1:=\mathcal{M}_0\otimes \Lambda V_1$.

A cocycle in $\mathcal{M}_0^1$ is of the form $a+v$ with $v\in V_0^1$, $a\in P^1$.
It follows from the injectivity of $H^1(P)\rightarrow H^1(A)$ that we may write $y_i=a_i+v_i$ such that the $v_i$ are linearly independent in $V_0^1$. Denote by $L_0\subset V_0^1$ the span of the $v_i$ and by $W_0$ a complement of $L_0$ in $V_0^{\geq 1}$. Let $\phi$ be an automorphism of $\mathcal{M}_1$ which is the identity on $P$, $W_0$, and $V_1$ and is defined as $v_i\mapsto v_i-a_i$ on $L_0$. Then $d'=\phi^{-1}d\phi$ defines a differential on $\mathcal{M}_1$ and $\phi$ is an isomorphism of cdgas with respect to the two differentials $d$ and $d'$. Now $d'$ maps $V_1^0$ isomorphically onto $L_0$ and we can write the inclusion $P\otimes \Lambda(V_1^0\oplus L_0)\rightarrow\mathcal{M}_1$ as a relative Sullivan algebra.
This implies that still $H^0(\mathcal{M}_1,d)=\mathbb{Q}$.

Furthermore we observe that regarding representatives of $\ker(\psi_1^*\colon H^1(\mathcal{M}_1)\rightarrow H^1(A))$ we are in the same situation as before: setting $\mathcal{M}_0'= \mathcal{M}_0\otimes \Lambda V_1^0$, a $d$-cocycle in $\mathcal{M}_1^1$ is of the form $a+v$ with $v\in V_1^1$, $a\in \mathcal{M}_0'$. Furthermore, if $a_i+v_i$ are representatives of a basis of $\ker (H^1(\mathcal{M}_1)\rightarrow H^1(A))$, then the $v_i$ are linearly independent in $V_1^1$. This is due to the fact that an element of the kernel which is represented in $\mathcal{M}_0'$ is also represented in $\mathcal{M}_0$ and is hence exact in $\mathcal{M}_1$ by construction.

We iterate this procedure of costructing the extension $\mathcal{M}_0\subset \mathcal{M}_1$ and obtain a factorization $P\rightarrow \mathcal{M}=\bigcup_{i\geq 0} \mathcal{M}_i\rightarrow A$ in which the first map is a free nilpotent extension and the second map is easily checked to be a quasi isomorphism.

Now this construction is functorial in the following sense: for a commutative diagram of solid arrows
\[\xymatrix{
P\ar[d]\ar[r]& \mathcal{M}\ar@{-->}[d]\ar[r]^{\sigma}& A\ar[d]^f\\
P'\ar[r]&\mathcal{M}'\ar[r]^{\sigma'}& A'
}\]
with $\mathcal{M}$, $\mathcal{M'}$ constructed as above, there is a canonical choice for the dashed arrow $\mathcal{M}\rightarrow \mathcal{M}'$ such that the diagram commutes and this choice is compatible with compositions. This can be seen inductively: given a map $\varphi\colon \mathcal{M}_i\rightarrow \mathcal{M}_i'$ which is compatible with the factorizations, we extend it to $\mathcal{M}_{i+1}$ by sending $v_{(x,y)}\in V^k_{i+1}$ to $v_{(f(x),\varphi(y))}\in V_{i+1}'$ whenever $k\geq 1$. For some $v_{(x,y)}\in V_{i+1}^0$, the element $v_{(f(x),\varphi(y))}$ is not necessarily in $V_{i+1}'$. However, there is some $v\in {\mathcal{M}_{i+1}^0}'$ with $dv=\varphi(y)$. Furthermore any element with this property is of the form $v+\alpha\cdot 1_{\mathcal{M}_i'}$ due to the fact that $H^0(\mathcal{M}_i')=\mathbb{Q}$. Setting $\varphi(v_{(x,y)})= v+\alpha\cdot 1_{\mathcal{M}_i'}$ we find that the element $f(\sigma(v_{(x,y)}))-\sigma'(\varphi(v_{(x,y)}))$ is a cocycle in ${A'}^0$. Thus there is a unique choice of $\alpha$ such that it vanishes, which is the unique admissible choice of $\varphi$ such that the diagram above commutes. This proves functoriality of the factorization.

Now given $P\rightarrow A$ where $P$ is cohomologically connected and $A$ is spacelike, we have a canonical product decomposition into morphisms $P\rightarrow A_i$ between cohomologically connected cochain algebras. We apply the above construction to each of these factors to obtain a factorization $P\rightarrow\mathcal{M}\rightarrow A$ in which the first map is a disconnected free nilpotent extension and the second map is a quasi isomorphism. Functoriality carries over to this case and yields the proposition.
\end{proof}

\begin{rem}
Applied to $\underline{\mathbb{Q}}\rightarrow \underline{A_{PL}(*)}$ we obtain a free nilpotent approximation of the system $\underline{A_{PL}(*)}$ of classifying spaces. We observe for later use that we in fact have an almost Sullivan approximation (parts of the following discussion are rather analogous to the proof of \cite[Theorem 14.7]{BibelI}): we apply the construction from the proof of Proposition \ref{prop:Sullivandiagram} to $\mathbb{Q}\rightarrow A_{PL}(B(T/U))$ for some subgroup $U\subset T$. This yields a quasi isomorphism $(\Lambda V,d)\rightarrow A_{PL}(B(T/U))$. With notation as in the proof of Proposition \ref{prop:Sullivandiagram}, we may rewrite $(\Lambda V,d)\cong ((\Lambda V^0\oplus L)\otimes \Lambda W,d')$ where $d'$ maps $V^0=\bigoplus_i V_i^0$ isomorphically onto $L=\bigoplus_i L_i\subset V^1$ and $W=W^{\geq 1}$ is a complement of $V^0\oplus L$ in $V$. Furthermore the inclusion
$(\Lambda(V^0\oplus L),d')\rightarrow (\Lambda(V^0\oplus L)\otimes \Lambda W,d')$ is a relative Sullivan algebra. It thus follows (cf.\ \cite[Proposition 6.7]{BibelI}) that dividing by the ideal generated by $(V^0\oplus L)$ yields a surjective quasi isomorphism
\[(\Lambda V,d)\rightarrow (\Lambda W,\overline{d'}).\]
By the lifting Lemma the morphism admits a section. Tensoring this section with the inclusion of $\Lambda(V^0\oplus L)\cong \Lambda (V^0\oplus dV^0)$ gives the desired isomorphism
\[\Lambda(W,\overline{d'})\otimes (\Lambda(V^0\oplus dV^0),d)\rightarrow(\Lambda V,d)\]
with $(\Lambda W,\overline{d'})$ a Sullivan algebra.
\end{rem}

The following technical Lemma will prove helpful at several occasions.

\begin{lem}\label{lem:tensorTC}
For $U\subset H\subset T$, consider a homotopy commutative diagram \[\xymatrix{P(H)\ar[r]\ar[d] & P(U)\ar[d]\\
A_{PL}(BT/H)\ar[r]& A_{PL}(BT/U)
}\]
where vertical maps are quasi isomorphisms and $P(H),P(U)$ are free nilpotent. Let $\mathcal{M}_H=(P(H)\otimes \Lambda V,d)$ be a free nilpotent extension and let $\mathcal{M}_U=P(U)\otimes_{P(H)} \mathcal{M}_H=(P(U)\otimes \Lambda V,d)$ be the pushout. Then the map
\[H^*(P(U))\otimes_{H^*(P(H))}H^*(\mathcal{M}_H)\rightarrow H^*(M_U)\]
is an isomorphism.
\end{lem}

\begin{proof}
We choose a minimal model $R_{T/H}\rightarrow A_{PL}(BT/H)$ which induces the identity on cohomology. Let $\varphi_H\colon P(H)\rightarrow R_{T/H}$ be a quasi isomorphism which is, up to homotopy, compatible with the chosen models. It admits a section $\psi_H$. Define $\varphi_U$ analogously. We note that the given map $\phi\colon P(U)\rightarrow P(H)$ is homotopic to $\phi'=\phi\circ\psi_H\circ\varphi_H$.
Now any homotopy between maps $f,g\colon P(H)\rightarrow P(U)\otimes (t,dt)$ induces morphisms
\[P(U)\otimes_{P(H)}\mathcal{M}_H\leftarrow (P(U)\otimes (t,dt))\otimes_{P(H)}\mathcal{M}_H\rightarrow P(U)\otimes_{P(H)}\mathcal{M}_H\]
where the outer tensor products are with respect to $f$ and $g$. By \cite[Theorem 6.10]{BibelI} these are quasi isomorphisms. Two more applications of the theorem yield that $p\colon \mathcal{M}_H\rightarrow R_{T/H}\otimes_{P(H)} \mathcal{M}_H=:\mathcal{M}_H'$ as well as
\[\phi_U\otimes p\colon P(U)\otimes_{P(H)}\mathcal{M}_H\rightarrow R_{T/U}\otimes_{R_{T/H}} \mathcal{M}_H'\]
 are quasi isomorphisms, where the first tensor product is taken with respect to $\phi'$ on the left and $R_{T/H}\rightarrow R_{T/U}$ is understood via $\varphi_U\circ\phi\circ\psi_H$. Clearly $H^*(R_{T/U}\otimes_{R_{T/H}} \mathcal{M}_H')= R_{T/U}\otimes_{R_{T/H}} H^*(\mathcal{M}_H')$.

\end{proof}

\begin{rem}\label{rem:technicalstuff}
Using Lemma \ref{lem:disconnected}, the theory surrounding free nilpotent cdgas carries over to the disconnected realm. We collect here a couple of facts which we will use throughout the paper. They follow immediately from the standard theory by considering individual components in the canonical product decompositions.

\begin{itemize}
\item Any morphism $B\rightarrow C$, with $B$ cohomologically connected and $C$ spacelike, factors as $B\rightarrow B\otimes \Lambda Z\rightarrow C$, where the first map is a disconnected Sullivan extension and the second map is a quasi isomorphism (cf.\ \cite[Theorem 3.1]{BibelII}).

\item Let $P$ be cohomologically connected. A morphism from a disconnected free nilpotent extension $P\rightarrow \mathcal{M}$ lifts through a surjective quasi isomorphism of spacelike cochain algebras relative to an initial lift of $P$. In case the quasi isomorphism is not surjective there is still a lift up to homotopy relative to $P$. Furthermore the homotopy class relative to $P$ of the lift is unique (cf.\ \cite[Lemma 14.4, Proposition 14.6]{BibelI}; note that results are formulated for the more restrictive relative Sullivan case, but the extra assumptions are not used; see also \cite[Lemma 1.2]{BibelII}).
\end{itemize}

For a cochain algebra $A$ the associated simplicial set, its so-called Sullivan realization, is denoted by $\langle A \rangle$ (cf.\ \cite[p.\ 247]{BibelI}). The spatial realization of $\langle A\rangle$ will be denoted by $\vert A\vert$. Note that directly from the definitions and Lemma \ref{lem:disconnected} we obtain $\langle \prod A_i\rangle=\coprod \langle A_i\rangle$ and $\vert \prod A_i\vert=\coprod \vert A_i\vert$. Thus the basic facts on realization carry over by considering path components separately. We will use:
\begin{itemize}
\item Let $\mathcal{M}$ be a disconnected free nilpotent algebra. Then the canonical map $A_{PL}(\vert\mathcal{M}\vert)\rightarrow A_{PL}(\langle\mathcal{M}\rangle)$ is a surjective quasi isomorphism (this is stated e.g.\ in \cite[p.\ 249]{BibelI} and \cite[p.\ 28]{BibelII} for Sullivan algebras and $\Lambda$-algebras respectively but the proofs apply equally in the more general setting).

\item Let $\mathcal{M}$ be an almost Sullivan algebra. If $H^1(\mathcal{M})=0$ and $\mathcal{M}$ is of finite type, then $\vert \mathcal{M}\vert$ is componentwise simply connected. Furthermore the canonical map $\mathcal{M}\rightarrow A_{PL}(\vert \mathcal{M}\vert)$ is a quasi isomorphism (the realization of an almost Sullivan cdga, i.e.\ a tensor product of a contractible cdga and a Sullivan cdga, splits into the product of the realization of the latter and a contractible space, as shown in \cite[Section 17 (c) Examples 1,2]{BibelI}; then use \cite[Theorem 17.10]{BibelI}).

\item If $P\rightarrow \mathcal{M}$ is a disconnected Sullivan extension, then the map $\vert\mathcal{M}\vert\rightarrow \vert P\vert$ is a fibration (cf.\cite[Proposition 17.9]{BibelI}).
\end{itemize}
\end{rem}
\subsection{The surjectivity condition}\label{sec:surjcond}

\begin{defn} Let $\mathcal{D}\subset\mathcal{S}$ be a subset.

\begin{itemize} 
\item For some $(U,H)\in\mathcal{D}$ let $\mathcal{D}(U,H)$ denote the subset of all $(U',H')>(U,H)$. Also denote by $\mathcal{D}^+(U,H)\subset \mathcal{D}(U,H)$ the subset of those $(U',H')>(U,H)$ with $U'\subsetneq U$ and by $\mathcal{D}^-(U,H)$ the subset of all $(U',H')>(U,H)$ with $H'\supsetneq H$.
\item
We say a $\mathcal{D}$-system $A$ satisfies the surjectivity condition \textbf{(SC)} if for any $(U,H)\in \mathcal{D}$ the map $\underline{A}(U,H)\rightarrow\lim_{\mathcal{D}(U,H)} \underline{A}$ is surjective.
\item For a $\mathcal{D}$-system $\underline{A}$ and $(U,H)\in\mathcal{D}$, we denote by $K_{\underline{A}}(U,H)$ the kernel of $\underline{A}(U,H)\rightarrow \lim_{\mathcal{D}(U,H)}\underline{A}$. We say a morphism $\underline{A}\rightarrow \underline{B}$ satisfies the surjectivity condition for morphisms \textbf{(SCM)} if $\underline{A}$ and $\underline{B}$ satisfy \textbf{(SC)} and the maps $K_{\underline{A}}(U,H)\rightarrow K_{\underline{B}}(U,H)$ are all surjective.
\end{itemize}
\end{defn}

\begin{prop}\label{prop:SCexistence} Let $\mathcal{D}\subset\mathcal{S}$ be bounded.
\begin{enumerate}[(i)]
\item  Let $\underline{A}$ be a $\mathcal{D}$-system. Then there is a free extension $\underline{A}\rightarrow \underline{B}$ which is pointwise given by taking the tensor product with a contractible free cdga, and  which satisfies the following properties: $\underline{B}$ satisfies \textbf{(SC)}. Furthermore, if $f\colon\underline{A}\rightarrow \underline{C}$ is a morphism of $\mathcal{D}$-systems and $\underline{C}$ satisfies \textbf{(SC)}, then $f$ extends to a morphism $\underline{B}\rightarrow\underline{C}$.
\item A given morphism $\underline{A}\rightarrow \underline{C}$ between two systems satisfying \textbf{(SC)} factors as $\underline{A}\rightarrow\underline{B}\rightarrow \underline{C}$ such that the first map is as in $(i)$ and the second map satisfies \textbf{(SCM)}.
\end{enumerate}
\end{prop}

\begin{proof}
In $(i)$, consider the filtration $\mathcal{D}=\bigcup_n \mathcal{D}_n$ where $\mathcal{D}_n$ consists of all those elements $(U,H)$ such that any maximal chain $(e,T)>\ldots> (U,H)$ is of length $\leq n$. We note that the filtration is exhaustive because $\mathcal{D}$ is bounded and that for $(U,H)\in \mathcal{D}_n$ we have $\mathcal{D}(U,H)\subset \mathcal{D}_{n-1}$. Assume for some $n$ that we have constructed free nilpotent extensions $\underline{A}|_{\mathcal{D}_n}\rightarrow \underline{B_n}$ of $\mathcal{D}_n$-systems, such that $\underline{B}_n$ satisfies \textbf{(SC)}. Now for every $(U,H)\in \mathcal{D}_{n+1}\backslash \mathcal{D}_{n}$ we define $\underline{B_{n+1}}(U,H)=\underline{A}(U,H)\otimes\Lambda(V\oplus dV)$ where $\Lambda(V\oplus dV)$ is a contractible free cdga which maps surjectively onto $\lim_{\mathcal{D}(U,H)} \underline{B_n}$. By construction, the map
\[\underline{B_{n+1}}(U,H)\rightarrow \lim_{\mathcal{D}(U,H)} \underline{B_{n+1}}=\lim_{\mathcal{D}(U,H)} \underline{B_n}\]
is surjective for $(U,H)\in \mathcal{D}_{n+1}$.
Iterating this procedure inductively yields the desired extension $\underline{A}\rightarrow \underline{B}$ in the limit. Now assume $f\colon\underline{A}\rightarrow \underline{C}$ is a morphisms and an extension of $f$ to $\underline{B_n}\rightarrow \underline{C}$ has been constructed. Then for $(U,H)$ and $\Lambda(V\oplus dV)$ as above we extend $f$ to $\underline{B_{n+1}}(U,H)=\underline{B_n}(U,H)\otimes\Lambda(V\oplus dV)$ by choosing any lift of the morphism
\[\Lambda(V\oplus dV)\rightarrow \lim_{\mathcal{D}(U,H)} \underline{B_n}\rightarrow \lim_{\mathcal{D}(U,H)} \underline{C}\] through the surjection $\underline{C}(U,H)\rightarrow \lim_{\mathcal{D}(U,H)} \underline{C}$. This extends $f$ to $\underline{B_{n+1}}$ as a morphism of diagrams.

For the proof of $(ii)$, at the $(U,H)$-position choose a contractible free cdga $\Lambda(W\oplus dW)$ which maps surjectively onto $K_{\underline{C}}(U,H)$. Then set $\underline{A'}(U,H)=\underline{A}(U,H)\otimes \Lambda(W\oplus dW)$ and map $\Lambda(W\oplus dW)$ to $0$ in $\underline{A'}(U',H')$ for any $(U',H')>(U,H)$. The obvious extension $\underline{A}'\rightarrow \underline{C}$ has the desired properties except $\underline{A'}$ does not necessarily satisfy \textbf{(SC)}. Now apply $(i)$.
\end{proof}

\begin{lem}\label{lem:pullbacksurj}
Let $\mathcal{D}\subset \mathcal{S}$ and let $\mathcal{F}\subset \mathcal{S}$ a subset with the property that every element of $\mathcal{F}$ is maximal in $\mathcal{D}\cup \mathcal{F}$.
Let $\underline{A}$ be a $\mathcal{D}\cup\mathcal{F}$-system satisfying \textbf{(SC)}. Then the projection $\lim_{\mathcal{D}\cup\mathcal{F}}\underline{A}\rightarrow \lim_{\mathcal{D}}\underline{A}$ is surjective.
\end{lem}

\begin{proof}
Let $x\in \lim_{\mathcal{D}}\underline{A}$. For any $(U,H)\in \mathcal{F}$ let $x_{(U,H)}$ denote the image of $x$ in $\lim_{\mathcal{D}(U,H)}\underline{A}$. Note that $\mathcal{D}(U,H)=(\mathcal{D}\cup\mathcal{F})(U,H)$ and let $y_{(U,H)}\in \underline{A}(U,H)$ be a preimage of $x_{(U,H)}$. Then a preimage of $x$ is given by the element $y\in \lim_{\mathcal{D}\cup\mathcal{F}}\underline{A}$ which is $y_{(U,H)}$ in the $(U,H)$-component for $(U,H)\in \mathcal{F}$ and agrees with $x$ on the other components.
\end{proof}

\begin{prop}\label{prop:SC}
Let $\mathcal{D}\subset \mathcal{S}$ be a bounded subset and let $\underline{A}\rightarrow \underline{B}$ be a quasi isomorphism of $\mathcal{D}$-systems which satisfy \textbf{(SC)}. Then $\lim_{\mathcal{D}}\underline{A}\rightarrow \lim_{\mathcal{D}}\underline{B}$ is a quasi isomorphism. If moreover $\underline{A}\rightarrow\underline{B}$ satisfies \textbf{(SCM)} then $\lim_{\mathcal{D}}\underline{A}\rightarrow \lim_{\mathcal{D}}\underline{B}$ is surjective.
\end{prop}

\begin{proof}
As $\mathcal{D}$ is bounded, any chain in $\mathcal{D}$ is of length $\leq n$ for some $n\geq 0$. Assume inductively that we have shown the proposition for diagrams whose maximal chain length is bounded by $n-1$.

Let $\mathcal{D}_{n-1}\subset \mathcal{D}$ be the subset of those $(U,H)\in \mathcal{D}$ such that any chain $(U,H)<\ldots $ in $\mathcal{D}$ is of length $\leq n-1$. By the induction hypothesis, $\lim_{\mathcal{D}_{n-1}}\underline{A}\rightarrow \lim_{\mathcal{D}_{n-1}}\underline{B}$ is a quasi isomorphism.
Observe that the requirements of Lemma \ref{lem:pullbacksurj} are fulfilled when writing $\mathcal{D}=\mathcal{D}_{n-1}\cup\mathcal{F}$ as a disjoint union.
By the lemma we obtain a diagram of complexes
\[\xymatrix{
0\ar[r] & K_{\underline{A}} \ar[r]\ar[d]& \lim_{\mathcal{D}}\underline{A}\ar[r]\ar[d]&\lim_{\mathcal{D}_{n-1}}\underline{A}\ar[r]\ar[d]&0\\
0\ar[r]& K_{\underline{B}}\ar[r]&\lim_{\mathcal{D}}\underline{B}\ar[r]&\lim_{\mathcal{D}_{n-1}}\underline{B}\ar[r]& 0
}\]
in which the rows are exact. Note that for $(U,H)\in \mathcal{F}$ we have $\mathcal{D}(U,H)\subset \mathcal{D}_{n-1}$. Thus the left hand map between the kernels is given by \[K_{\underline{A}}\cong \prod_{(U,H)\in\mathcal{F}}K_{\underline{A}}(U,H)\rightarrow \prod_{(U,H)\in\mathcal{F}}K_{\underline{B}}(U,H)\cong K_{\underline{B}}\]
with $K_{\underline{A}}(U,H)$ and $K_{\underline{B}}(U,H)$ as defined above.
We claim that $K_{\underline{A}}(U,H)\rightarrow K_{\underline{B}}(U,H)$ induces an isomorphism in cohomology for any $(U,H)\in\mathcal{D}$. If this holds, then the same is true for $K_{\underline{A}}\rightarrow K_{\underline{B}}$ and thus also the middle vertical morphism in the diagram above, which finishes the induction. To prove the claim consider the diagram
\[\xymatrix{
0\ar[r] & K_{\underline{A}}(U,H) \ar[r]\ar[d]& \underline{A}(U,H)\ar[r]\ar[d]&\lim_{\mathcal{D}(U,H)}\underline{A}\ar[r]\ar[d]&0\\
0\ar[r]& K_{\underline{B}}(U,H)\ar[r]&\underline{B}(U,H)\ar[r]&\lim_{\mathcal{D}(U,H)}\underline{B}\ar[r]& 0
}\]
with exact rows. The central vertical map is a quasi isomorphism by hypothesis and the right hand vertical map is a quasi isomorphism by the induction hypothesis. Thus the claim and the first part of the proposition follow. If $K_{\underline{A}}(U,H)\rightarrow K_{\underline{B}}(U,H)$ is surjective for all $(U,H)\in \mathcal{D}$, then in the induction above $K_{\underline{A}}\rightarrow K_{\underline{B}}$ is surjective. Assuming inductively that $\lim_{\mathcal{D}_{n-1}}\underline{A}\rightarrow \lim_{\mathcal{D}_{n-1}}\underline{B}$ is surjective, then the first diagram above implies surjectivity of $\lim_{\mathcal{D}_{n}}\underline{A}\rightarrow \lim_{\mathcal{D}_{n}}\underline{B}$
\end{proof}

\begin{prop}\label{prop:lifting}
Let $\mathcal{D}\subset \mathcal{S}$ be a bounded subset.
Consider a commutative diagram of solid arrows of $\mathcal{D}$-systems
\[\xymatrix{
\underline{P}\ar[d]\ar[r] & \underline{A}\ar[d]\\
\underline{\mathcal{M}}\ar[r]\ar@{-->}[ur]& \underline{B}
}\]
in which the left hand vertical map is a disconnected free nilpotent extension and the right hand vertical map is a quasi isomorphism between spacelike systems satisfying \textbf{(SCM)}. Then the dashed arrow exists such that the diagram commutes.
\end{prop}

\begin{proof}
Let $\mathcal{D}=\bigcup_n \mathcal{D}_n$ as in the proof of Proposition \ref{prop:SC}. Assume that we have constructed the lift when restricting to $\mathcal{D}_{n-1}$-systems. Let $(U,H)\in \mathcal{D}_n\backslash \mathcal{D}_{n-1}$.  Consider the pullback square
\[\xymatrix{
Q\ar[r]\ar[d]&\lim_{\mathcal{D}(U,H)} \underline{A}\ar[d]\\
\underline{B}(U,H)\ar[r]& \lim_{\mathcal{D}(U,H)}\underline{B}
}\]
with the compatible $\underline{P}(U,H)$-actions. By Proposition \ref{prop:SC} the right hand vertical map is a surjective quasi isomorphism so the same holds for $Q\rightarrow \underline{B}(U,H)$. In particular it follows that $\underline{A}(U,H)\rightarrow Q$ is a quasi isomorphism. It follows from the \textbf{(SCM)} condition that $\underline{A}(U,H)\rightarrow Q$ is also surjective. Note that $\mathcal{D}(U,H)\subset \mathcal{D}_{n-1}$. Thus by induction we have a map $\underline{\mathcal{M}}(U,H)\rightarrow Q$, compatible with the $\underline{P}(U,H)$-actions. We lift this relative to $\underline{P}(U,H)$ to a map $\underline{\mathcal{M}}(U,H)\rightarrow \underline{A}(U,H)$. This defines the desired lift on $\mathcal{D}_n$ and finishes the proof.
\end{proof}

\begin{prop}\label{prop:technicaldiagramstuff}
Let $\mathcal{D}\subset \mathcal{S}$ be a stable bounded subset and $\underline{A}$ be a $\mathcal{D}$-system satisfying \textbf{(SC)}. Let $(U,H)\in \mathcal{D}$.
\begin{enumerate}[(i)]
\item The projections
\[\lim_{\mathcal{D}(U,H)}\underline A\longrightarrow \lim_{\mathcal{D}^+(U,H)}\underline{A}\]
and
\[\lim_{\mathcal{D}(U,H)}\underline A\longrightarrow \lim_{\mathcal{D}^-(U,H)}\underline{A}\]
are surjective.
\item
If $H^0(\underline{A}(U',H))\rightarrow H^0(\underline{A}(\{1\},H))$ is an isomorphism and $H^1(\underline{A}(U',H))=0$ for any $(U,H)\leq (U',H)\in \mathcal{D}$ then $H^1(\lim_{\mathcal{D}^+(U,H)}\underline{A})=0$. 
\item
If in addition to the conditions of $(ii)$ the map $H^2(\underline{A}(U',H))\rightarrow H^2(\underline{A}(\{1\},H))$ is injective for every $U'\subsetneq U$, then $H^2(\lim_{\mathcal{D}^+(U,H)}\underline{A})\rightarrow H^2(\underline{A}(e,H))$ is injective.
\end{enumerate}
\end{prop}

\begin{proof}
We prove only the surjectivity of the first map in $(i)$. The proof for the second map is completely analogous with the roles of the two components swapped.
We have $\mathcal{D}(U,H)=\mathcal{D}^+(U,H)\cup F$, where $F=\{(U,H')\in \mathcal{D}~|~ H\subsetneq H'\}$. We filter $F=\bigcup F_n$ where $F_n$ consists of those $(U,H')$ such that the length of any maximal chain $(U,T)>\ldots>(U,H')$ in $\mathcal{D}$ is of length $\leq n$. The map \[\lim_{\mathcal{D}^+(U,H)\cup F_{n+1}}\underline{A}\longrightarrow \lim_{\mathcal{D}^+(U,H)\cup F_n}\underline{A}
\]
is surjective by Lemma \ref{lem:pullbacksurj} and the surjectivity of $\lim_{\mathcal{D}(U,H)}\underline A\longrightarrow \lim_{\mathcal{D}^+(U,H)}\underline{A}$ follows inductively.

Now suppose $H^0(\underline{A}(U,H))\rightarrow H^0(\underline{A}(U',H))$ is an isomorphism and $H^1(\underline{A}(U',H))=0$ for any $U\supset U'\in \mathcal{D}_L$. For any $(K,L)\in \mathcal{S}$, let $\overline{\mathcal{D}(K,L)}=\{(K',L)\in\mathcal{D}~|~K'\subsetneq K\}$. We claim that the projection
\[\lim_{\mathcal{D}^+(U,H)}\underline{A}\longrightarrow \lim_{\overline{\mathcal{D}(U,H)}}\underline{A}\]
is an isomorphism. Any element $(K,L)\in\mathcal{D}^+(U,H)$ has the predecessor $(K,H)\in \overline{D(U,H)}$. Thus the projection map is an injection. On the other hand every element of $\lim_{\overline{\mathcal{D}(U,H)}}\underline{A}$ extends to one of $\lim_{{\mathcal{D}^+(U,H)}}\underline{A}$ by setting the $(K,L)$ component to be the image of the $(K,H)$ component.

We show now that $H^1(\lim_{\overline{\mathcal{D}(U,H)}}\underline{A})=0$. Define a filtration $\overline{\mathcal{D}(U,H)}=\bigcup_n E_n$ where $E_n$ consists of those $(U',H)$ where the length of any maximal chain $(\{1\},H)>\ldots>(U',H)$ in $\mathcal{D}$ is $\leq n$. We claim that $\lim_{E_{n+1}}\underline{A}\rightarrow \lim_{E_n}\underline{A}$ is surjective. The composition
\[\underline{A}(K,H)\rightarrow\lim_{\mathcal{D}(K,H)}\underline{A}\rightarrow \lim_{\mathcal{D}^+(K,H)}\underline{A}\cong \lim_{\overline{\mathcal{D}(K,H)}}\underline{A}\]
is surjective by \textbf{(SC)} and the first part of the proposition. Thus $\underline{A}|_{\overline{D(U,H)}}$ satisfies \textbf{(SC)} and the claim follows from Lemma \ref{lem:pullbacksurj}.

The algebra $\lim_{E_0}\underline{A}=\underline{A}(e,H)$ is simply connected. Now assume inductively that $\lim_{E_n}\underline{A}$ is simply connected. Let $B$ be the kernel fitting in the exact sequence 
\[0\longrightarrow B\longrightarrow \lim_{E_{n+1}}\underline{A}\longrightarrow \lim_{E_n}\underline{A}\longrightarrow 0\]
of complexes. By the long exact cohomology sequence and the induction hypothesis $H^1(B)\rightarrow H^1(\lim_{E_{n+1}}\underline{A})$ is surjective. We observe that 
\[B\cong \prod_{(K,H)\in E_{n+1}\backslash E_n} B_{(K,H)}\]
as complexes, where
\[0\longrightarrow B_{(K,H)}\longrightarrow \underline{A}(K,H)\longrightarrow \lim_{\overline{\mathcal{D}(K,H)}}\underline{A}\longrightarrow 0\]
is exact. It suffices to show that all the $B_{(K,H)}$ are cohomologically simply connected. As $H^1(\underline{A}(K,H))=0$ it suffices to show that $H^0(\underline{A}(K,H))\rightarrow H^0(\lim_{\overline{D(K,H)}}\underline{A})$ is an isomorphism.
To see this, recall that by hypothesis $H^0(\underline{A}(U',H))\rightarrow H^0(\underline{A}(\{1\},H))$ is an isomorphism for all $U\supset U'\in\mathcal{D}_L$. Since all algebras are nonnegatively graded, degree $0$ cohomology is just the kernel of the differentials. But then clearly the projection $\lim_{\overline{\mathcal{D}(K,H)}}\underline{A}\rightarrow \underline{A}(\{1\},H)$ maps the degree $0$ kernels of the differentials isomorphically to each other.

For the proof of $(iii)$ we note that $\lim_{E_0}\underline{A}=\underline{A}(\{1\},H)$ and proceed by showing that the maps $H^2(\lim_{E_{n+1}}\underline{A})\rightarrow H^2(\lim_{E_n}\underline{A})$ are injective for all $n$. For this it suffices that $H^2(B)=0$, which is equivalent to $H^2(B_{(K,H)})=0$ for all $(K,H)\in E_{n+1}\backslash E_n$. For any such $(K,H)$, the map $H^2(\underline{A}(K,H))\rightarrow H^2(\lim_{\overline{\mathcal{D}(K,H)}}\underline{A})$ is injective because the composition with the projection onto $H^2(\underline{A}(\{1\},H))$ is. But then $H^2(B_{(K,H)})=0$ follows from the long exact sequence of $0\longrightarrow B_{(K,H)}\longrightarrow \underline{A}(K,H)\longrightarrow \lim_{\overline{\mathcal{D}(K,H)}}\underline{A}\longrightarrow 0$ using that $H^1(\lim_{\overline{\mathcal{D}(K,H)}}\underline{A})=0$ by $(ii)$.
\end{proof}

\begin{lem}\label{lem:subcomplexes}
Let $X$ be a $T$-CW-complex with isotropies contained in a bounded collection $C$ of subgroups of $T$. Set $\mathcal{D}=\mathcal{D}(C)$. For some $H\in C$ we define $\mathcal{D}^-(H)\subset \mathcal{D}$ to be the subset of elements of the form $(H_0,K)$ with $K\supsetneq H$. Also set $Y$ to be the union over all $X^K$ for $T\supset K\supsetneq H$. Then $\underline{A_{PL}(X)}|_{\mathcal{D}^-(H)}$ satisfies \textbf{(SC)} and $A_{PL}(Y_{T/H_0})\rightarrow \lim_{\mathcal{D}^-(H)} \underline{A_{PL}(X)}$ is a surjective quasi isomorphism.
\end{lem}

\begin{proof}
We begin with the surjectivity of $A_{PL}(Y_{T/H_0})\rightarrow \lim_{\mathcal{D}^-(H)} \underline{A_{PL}(X)}$. In what follows, $Sing(-)$ denotes the simplicial set given by the singular simplices of a space.
The map corresponds to the morphism $A_{PL}(Sing(Y_{T/H_0}))\rightarrow A_{PL}(\colim Sing(X_{T/H_0}^K))$ induced by the morphism of simplicial sets $\colim Sing(X_{T/H_0}^K))\rightarrow Sing(Y_{T/H_0})$ and the colimit is formed over all $(H_0,K)\in\mathcal{D}^-(H)$.  For $(H_0,K)$, $(H_0,K')\in\mathcal{D}^-(H)$ we have $X_{T/H_0}^K\cap X_{T/H_0}^{K'}=X_{T/H_0}^{K''}$ for some $(H_0,K'')\in\mathcal{D}^-(H)$.
From this we deduce that the colimit can be identified with the subset of singular simplices which respect the cover $Y_{T/H_0}=\bigcup_{(H_0,K)\in \mathcal{D}^-(H)} X_{T/H_0}^K$. The inclusion of simplicial sets induces a surjection on $A_{PL}(-)$ (see \cite[Proposition 10.4]{BibelI}). The same argument shows that $\underline{A_{PL}(X)}|_{\mathcal{D}^-(H)}$ satisfies \textbf{(SC)}.

To show that $A_{PL}(Y_{T/H_0})\rightarrow \lim_{\mathcal{D}^-(H)} \underline{A_{PL}(X)}$ is a quasi isomorphism, we first replace the $X_{T/H_0}^K$ by open neighbourhoods. For a fixed $0<\varepsilon<1$ and $T$-subcomplex $A\subset X$ we define an equivariant open neighbourhood as follows: inductively assume we have constructed a neighbourhood $U_n$ of $A^n$ in $X^n$. Then for each $(n+1)$-cell $\phi\colon T/K\times D^{n+1}\rightarrow X^n$ which is not in $A^{n+1}$ we radially thicken up $\phi^{-1}(U_n)$ by length $\varepsilon$ into the interior of $T/K\times D^{n+1}$. This construction gives neighbourhoods $U_{(H_0,K)}\subset Y_{T/H_0}$ for every $(H_0,K)\in\mathcal{D}^-(H)$ which deformation retract onto $X_{T/H_0}^K$. For any $(H_0,K)\leq (H_0,K')$ we have $U_{(H_0,K)}\supset U_{(H_0,K')}$ and  we have $U_{(H_0,K)}\cap U_{(H_0,K')}=U_{(H_0,K'')}$ whenever $X_{T/H_0}^K\cap X_{T/H_0}^{K'}=X_{T/H_0}^{K''}$. Then $\lim_{\mathcal{D}^-(H)} \underline{A_{PL}(X)}\rightarrow \lim_{(H_0,K)\in\mathcal{D}^-(H)} A_{PL}(U_{(H_0,K)})$ is a quasi isomorphism since this holds pointwise and both systems satisfy \textbf{(SC)}. As before, the right hand algebra can be understood as $A_{PL}(-)$ applied to the subset of $Sing(Y_{T/H_0})$ which respects the open cover given by the $U_{(H_0,K)}$.

In order to show that $A_{PL}(Y_{T/H_0})\rightarrow \lim_{(H_0,K)\in\mathcal{D}^-(H)} A_{PL}(U_{(H_0,K)})$ is a quasi isomorphism, by \cite[Theorem 10.9]{BibelI}, it suffices to prove that the inclusion of the corresponding simplicial sets induces a quasi isomorphism on the standard singular cochain algebras. This is a standard result which follows via barycentric subdivision.
\end{proof}

\section{Strictification}
The goal of this section is to show that the notion of homotopy of systems reduces to pointwise considerations in the specific case of models for the Borel fibration on $T$-simply connected spaces.

By $(t,dt)$ we denote the free contractible cdga on the generators $t$ in degree $0$ and $dt$ in degree $1$, where the differential sends $t$ to $dt$. There are two evaluation morphisms $i_0$ and $i_1$ to $\mathbb{Q}$ given by $t\mapsto 0$ and $t\mapsto 1$ respectively and $dt\mapsto 0$.
Recall that a homotopy between two morphisms $f,g\colon A\rightarrow B$, is a morphism $H\colon A\rightarrow B\otimes (t,dt)$ with $i_0\circ H=f$, $i_1\circ H=g$, where we write $i_j$ for the map $\id_B\otimes i_j$ by abuse of notation. For a system $\underline{A}$ of cdgas we form the system $\underline{A\otimes (t,dt)}$ by taking the pointwise tensor product and extending maps via the identity on the right hand side. We obtain morphisms $i_0,i_1\colon \underline{A\otimes(t,dt)}\rightarrow \underline{A}$ of systems. We point out that if $\mathcal{D}\subset \mathcal{S}$ is finite, then $\lim_{\mathcal{D}}\underline{A\otimes(t,dt)}=\lim_\mathcal{D}\underline{A}\otimes (t,dt)$.

\begin{defn} A homotopy between morphisms $f,g\colon\underline{A}\rightarrow\underline{B}$ of systems of cdgas is a morphism  $H\colon \underline{A}\rightarrow \underline{B\otimes (t,dt)}$ satisfying $i_0\circ H=f$ and $i_1\circ H=g$.
\end{defn}

The following standard lemma is the key to passing from homotopy commutative diagrams to strictly commutative ones.

\begin{lem}\label{lem:homotopetocom}
Let $P$ be cohomologically connected and $A,B$ be spacelike cochain algebras. Consider a diagram
\[\xymatrix{
 P\ar[r]\ar[d]& A\ar[d]^\pi\\
 \mathcal{M}\ar[ur]^\varphi\ar[r]^\psi & B
}\]
in which $P\rightarrow \mathcal{M}$ is a disconnected free nilpotent extension, the outer square as well as the upper left triangle commute, and the lower right triangle commutes up to homotopy relative to $P$. If $\pi$ is surjective, then we find a morphism $\varphi'$ homotopic to $\varphi$ relative to $P$, such that $\psi=\pi\circ\varphi'$.
\end{lem}

\begin{proof}
Let $H\colon \mathcal{M}\rightarrow B\otimes (t,dt)$ be a homotopy with $i_0\circ H=\psi$, $i_1\circ H=\pi\circ\varphi$. Denote by $Q$ the pullback of $B\otimes (t,dt)\xrightarrow{i_1} B\xleftarrow{\pi} A$. Then we have a map $(\pi\otimes \id_{(t,dt)})\times i_1\colon A\otimes (t,dt)\rightarrow Q$. One checks that this is a surjective quasi isomorphism.
Now consider the map $H\times \varphi\colon \mathcal{M}\rightarrow Q$. We lift this through $(\pi\otimes \id_{(t,dt)})\times i_1$ relative to $P$, where we use the constant homotopy $P\rightarrow A\otimes (t,dt)$ as an initial lift.
The result is a homotopy $\tilde{H}\colon \mathcal{M}\rightarrow A\otimes (t,dt)$ relative to $P$ which satisfies $i_1\circ \tilde{H}=\varphi$. On the other hand $\pi\circ i_0\circ \tilde{H}=i_0\circ (\pi\otimes \id_{(t,dt)})\circ \tilde{H}=i_0\circ H=\psi$
\end{proof}

\begin{lem}\label{lem:cohomohomotopy}
Let $T$ be a torus, $P\rightarrow A_{PL}(BT)$ a free nilpotent model and $A$ any cochain algebra. Then two maps $P\rightarrow A$ are homotopic if and only if they induce the same map on cohomology.
\end{lem}

\begin{proof}
Let $f,g\colon P\rightarrow A$ be two maps which induce the same map on cohomology.
Choose a minimal model $\varphi\colon(\Lambda (x_1,\ldots,x_r),0)\rightarrow P$ with $x_i$ generators of degree $2$. It suffices to show that $f\circ\varphi$ and $g\circ\varphi$ are homotopic since composition with $\varphi$ induces a bijection on homotopy classes. By assumption there are $y_1,\ldots,y_r\in A^1$ such that $f(\varphi(x_i))-g(\varphi(x_i))=dy_i$. Then the desired homotopy is defined by $H\colon R\rightarrow A\otimes (t,dt)$, $x_i\mapsto f(\varphi(x_i))+(g(\varphi(x_i))-f(\varphi(x_i)))t+y_idt$.
\end{proof}

\begin{prop}\label{prop:coveringhomotopy}
Let $P\rightarrow A_{PL}(BT)$ be a free nilpotent model. Let

\[\xymatrix{
 & A\ar[d]^\pi\\ P\ar[ur]^{\tilde{f}}\ar[r]^f & B
}
\quad and \quad
\xymatrix{
 & A\ar[d]^\pi\\ P\ar[ur]^{\tilde{g}}\ar[r]^g & B
}\]
be two commutative diagrams of spacelike cochain algebras in which $\pi$ is surjective and $H^1(B)=0$. Suppose further that $H\colon P\rightarrow B\otimes (t,dt)$ is a homotopy from $f$ to $g$ and that $\tilde{f}$, $\tilde{g}$ are homotopic as well. Then there is a homotopy $\tilde{H}\colon P\rightarrow A\otimes (t,dt)$ from $\tilde{f}$ to $\tilde{g}$ such that ${H}=\pi\otimes \id_{(t,dt)}\circ \tilde{H}$.
\end{prop}

\begin{proof}
Consider the commutative diagram

\[\xymatrix{P\ar@/_1.0pc/[ddr]_{\tilde{f}\times\tilde{g}}\ar@/^1.0pc/[rrd]^H & &\\
 & A\otimes (t,dt)\ar[d]\ar[r] & B\otimes (t,dt)\ar[d]\\
 & A\times A\ar[r] & B\times B
}\]
and let $G$ be some homotopy from $\tilde{f}$ to $\tilde{g}$. Let $Q$ denote the pullback of $A\times A\rightarrow B\times B\leftarrow B\otimes (t,dt)$. Then the above diagram induces the diagram

\[\xymatrix{
 & A\otimes (t,dt)\ar[d]^\alpha\\
 P\ar[ru]^G\ar[r]^\beta & Q
}\]
which we claim to be homotopy commutative. If this is true then by Lemma \ref{lem:homotopetocom}, we can homotope $G$ to a map $\tilde{H}$ which makes the diagram strictly commutative and thus gives the desired homotopy.

Let $x\in P^2$ be closed. Writing $Q\subset (A\times A)\times (B\otimes (t,dt))$, we find that the difference $\beta-\alpha\circ G$ maps $x$ to an element $((0,0),y)$ where $y\in \ker(B\otimes (t,dt)\rightarrow B\times B)$. It follows that $y=dz$ is exact. Let $z'\in A^1\times A^1$ be an element whose image in $B\times B$ agrees with that of $z$. It follows that $dz'\in K=\ker(A\times A\rightarrow B\times B)$. The long exact cohomology sequence of $K\rightarrow A\times A\rightarrow B\times B$ together with the fact that $H^1(B)=0$ implies that $H^2(K)\rightarrow H^2(A\times A)$ is injective. Thus there is some $a\in K^1$ with $da=dz'$. The element $(z'-a,z)$ lies in $Q$ and thus $\beta(x)-\alpha\circ G(x)$ is exact in $Q$. The claim now follows from Lemma \ref{lem:cohomohomotopy}.
\end{proof}

\begin{defn}
A weak morphism $f\colon \underline{A}\rightarrow \underline{B}$ of $\mathcal{D}$-systems is a morphism $\underline{A}(U,H)\rightarrow \underline{B}(U,H)$ for every $(U,H)\in \mathcal{D}$, such that the diagrams
\[\xymatrix{
\underline{A}(U,H)\ar[r]\ar[d] & \underline{B}(U,H)\ar[d]\\ \underline{A}(K,L)\ar[r]& \underline{B}(K,L)}\]
commute up to homotopy. We also refer to such a morphism as a weak $\underline{A}$-structure on $\underline{B}$. If the diagrams commute strictly (i.e.\ $f$ is a morphism of $\mathcal{D}$-systems) we also refer to $f$ as a strict morphism and say $\underline{B}$ has a (strict) $\underline{A}$-structure.
\end{defn}

\begin{rem}\label{rem:weakvscohomology}
There is a weak morphism $\underline{R}\rightarrow \underline{A_{PL}(*)}$ which is a pointwise minimal Sullivan model. Thus if $\underline{P}\rightarrow\underline{A_{PL}(*)}$ is a weak quasi isomorphism which is pointwise a free nilpotent model, then there are weak quasi isomorphisms $\underline{P}\simeq\underline{R}$ in both directions. This establishes a correspondence between weak $\underline{P}$-structures and cohomology $\underline{R}$-structures on $\underline{A}$ in the following fashion.

Clearly every weak morphism $\underline{P}\rightarrow\underline{A}$ induces a cohomology $\underline{R}$-structure via $H^*(\underline{P})\cong\underline{R}$. Conversely, by Lemma \ref{lem:cohomohomotopy}, choosing cocycle representatives produces a weak morphism $\underline{R}\rightarrow \underline{A}$ from a cohomology $\underline{R}$-structure on $\underline{A}$. Composing with the weak morphism $\underline{P}\rightarrow \underline{R}$ gives a weak morphism $\underline{P}\rightarrow\underline{A}$. These two constructions are inverse to one another up to pointwise homotopy on the side of weak $\underline{P}$-structures.
We extend the conditions \textbf{(TC)} and \textbf{(LC)} to (weak) $\underline{P}$-structures in the obvious fashion.
\end{rem}

\begin{thm}\label{thm:strictification}
Let $\mathcal{D}$ be stable and finite.
Let $\underline{P}$ be a free nilpotent $\mathcal{D}$-system satisfying $\underline{P}(U,H)=\underline{P}(U,T)$ for all $(U,H)\in \mathcal{D}$ and which admits a weak quasi isomorphism $\underline{P}\rightarrow \underline{A_{PL}(*)}$. Let $\underline{A}$ be a spacelike $\mathcal{D}$-system satisfying $H^1(\underline{A})=0$, \textbf{(SC)}, and \textbf{(TC)}.
\begin{enumerate}[(i)]
\item For every weak morphism $\underline{P}\rightarrow \underline{A}$, there is a pointwise homotopic strict morphism $\underline{P}\rightarrow\underline{A}$.
\item For any two strict morphisms $f,g\colon\underline{P}\rightarrow \underline{A}$ which are pointwise homotopic to one another, there is a strict morphism $H\colon \underline{P}\rightarrow \underline{A\otimes (t,dt)}$ which is a homotopy from $f$ to $g$.
\item Let $\underline{P}\rightarrow\underline{\mathcal{M}}$ be a disconnected free nilpotent extension and $f,g\colon \underline{\mathcal{M}}\rightarrow \underline{A}$ be two morphisms whose restrictions to $\underline{P}$ agree. Then if $f$ and $g$ are homotopic, they are connected by a chain of homotopies relative to $\underline{P}$.
\end{enumerate}
\end{thm}

\begin{proof}
Let $f\colon \underline{P}\rightarrow \underline{A}$ a weak morphism. We construct a strict morphism $g\colon \underline{P}\rightarrow\underline{A}$ which is pointwise homotopic to $f$.
We filter $\mathcal{D}=\bigcup_n \mathcal{D}_n$ where $\mathcal{D}_n$ consists of those $(U,H)\in \mathcal{D}$ where $U$ is of dimension $\leq n$.

Assume we have constructed $g\colon\underline{P}|_{\mathcal{D}_n}\rightarrow \underline{A}|_{\mathcal{D}_n}$ which is pointwise homotopic to $f$. Let $U\in \mathcal{D}_L$ of dimension $n+1$ and recall that $m_\mathcal{D}(U)$ is the intersection of all groups in $\mathcal{D}_R$ which contain $U$. Thus $(U,m_\mathcal{D}(U))\in \mathcal{D}$ and $(U,m_\mathcal{D}(U))\leq (U,H')$ for any $(U,H')\in \mathcal{D}$. We have $\mathcal{D}^+(U,m_\mathcal{D}(U))\subset \mathcal{D}_n$
and a diagram

\[\xymatrix{
 & \underline{A}(U,m_\mathcal{D}(U))\ar[d]^\pi\\
 \underline{P}(U,m_\mathcal{D}(U))\ar[r]^{\tilde{g}}\ar[ur]^{f_{(U,m_\mathcal{D}(U))}}&\lim_{\mathcal{D}^+(U,m_\mathcal{D}(U))}\underline{A}
}\]
 which we claim commutes up to homotopy. By assumption the two maps are homotopic after composing with the map to $\underline{A}(\{1\},m_\mathcal{D}(U))$. Note that \textbf{(TC)} assures that the requirements of Proposition \ref{prop:technicaldiagramstuff} $(i)$, $(ii)$, and $(iii)$ are all satisfied. Thus $H^2(\lim_{\mathcal{D}^+(U,m_\mathcal{D}(U))}\underline{A})\rightarrow H^2(\underline{A}(\{1\},m_\mathcal{D}(U)))$ is injective and it follows that $\pi\circ f$ and $\tilde{g}$ induce the same map on cohomology. By Lemma \ref{lem:cohomohomotopy} this implies homotopy commutativity. Then by Lemma \ref{lem:homotopetocom} we can homotope $f_{(U,m_\mathcal{D}(U))}$ to a map $g_{(U,m_\mathcal{D}(U))}$ which makes the diagram strictly commutative. For any $(U,H)\in\mathcal{D}_{n+1}$ we define $g_{(U,H)}$ as the composition $\underline{P}(U,H)=\underline{P}(U,m_\mathcal{D}(U))\rightarrow \underline{A}(U,m_\mathcal{D}(U))\rightarrow \underline{A}(U,H)$. Doing this for all $U$ completes the induction and the proof of $(i)$.
 
The proof of $(ii)$ proceeds via induction over the $\mathcal{D}_n$ as well. Suppose we have a strict morphism $H\colon\underline{P}|_{\mathcal{D}_n}\rightarrow \underline{A\otimes (t,dt)}|_{\mathcal{D}_n}$ which is pointwise a homotopy between $f$ and $g$.
Let $U\in\mathcal{D}_L$ be a fixed $(n+1)$-dimensional subtorus. Then $H$ induces the second map in the composition
\[\phi_U\colon\underline{P}(U,m_\mathcal{D}(U))\rightarrow \lim_{\mathcal{D}^+(U,m_\mathcal{D}(U))} \underline{P}\rightarrow\lim_{\mathcal{D}^+(U,m_\mathcal{D}(U))} \underline{A}\otimes(t,dt)\]
and defines a homotopy between the two maps $\underline{P}(U,m_\mathcal{D}(U))\rightarrow \lim_{\mathcal{D}^+(U,m_\mathcal{D}(U))} \underline{A}$ induced by $f$ and $g$.
By Proposition \ref{prop:coveringhomotopy} we find a homotopy $H_{U}\colon \underline{P}(U,m_\mathcal{D}(U))\rightarrow \underline{A}(U,m_\mathcal{D}(U))$ between the $(U,m_\mathcal{D}(U))$-positions of $f$ and $g$, such that $(\pi\otimes \id_{(t,dt)})\circ H_U=\phi_U$, where $\pi$ is the map $\underline{A}(U,m_\mathcal{D}(U))\rightarrow \lim_{\mathcal{D}^+(U,m_\mathcal{D}(U))} \underline{A}$. This induces homotopies $\underline{P}(U,H)=\underline{P}(U,m_\mathcal{D}(U))\rightarrow \underline{A}(U,m_\mathcal{D}(U))\otimes (t,dt)\rightarrow \underline{A}(U,H)\otimes (t,dt)$ for all $(U,H)\in\mathcal{D}$. Doing this for all $U\in\mathcal{D}_L$ of dimension $n+1$ we can extend the homotopy $H$ to a map $\underline{P}|_{\mathcal{D}_{n+1}}\rightarrow \underline{A\otimes (t,dt)}|_{\mathcal{D}_{n+1}}$ which finishes the induction.

Now for the proof of $(iii)$, let $G\colon \mathcal{M}\rightarrow \underline{A\otimes (t,dt)}$ be a homotopy from $f$ to $g$. The first step is to homotope $G|_{\underline{P}}$ relative to the endpoints to a constant homotopy. Consider two models for the interval $(s,ds)$ and $(t,dt)$ together with the usual evaluation maps $e_{s=0},e_{s=1}\colon (s,ds)\rightarrow \mathbb{Q}$ and $e_{t=0},e_{t=1}\colon (t,dt)\rightarrow \mathbb{Q}$. Consider the subalgebra \[S\subset (t,dt)\times (s,ds)\times (t,dt)\times (s,ds)\]
consisting of those $(a,b,c,d)$ with $e_{t=1}(a)=e_{s=0}(b)$, $e_{s=1}(b)=e_{t=1}(c)$, $e_{t=0}(c)=e_{s=1}(d)$, and $e_{s=0}(d)=e_{t=0}(a)$. This a model for the boundary of the square, i.e. for $S^1$, with the four factors corresponding to the line segments. For $i=0,1$, let $q_{s=i}$ (resp.\ $q_{t=i}$) denote the map $e_{s=i}\otimes \id\colon (s,ds)\otimes (t,dt)\rightarrow (t,dt)$ (resp.\ $\id\otimes e_{t=i}\colon (s,ds)\otimes (t,dt)\rightarrow (s,ds)$). The map $q_{s=0}\times q_{t=1}\times q_{s=1}\times q_{t=0}$ defines a surjection $(s,ds)\otimes (t,dt)\rightarrow S$ which corresponds to the inclusion of the boundary into the square. The $q_{s=i}$, $q_{t=i}$ factor through $S$. The induced maps $\underline{A\otimes S}\rightarrow\underline{A\otimes (t,dt)}$ and $\underline{A\otimes S}\rightarrow\underline{A\otimes (s,ds)}$ are denoted ${p}_{s=i}$ and ${p}_{t=i}$.

There is a unique map $F\colon\underline{P}\rightarrow \underline{A\otimes S}$ such that $p_{s=0}\circ F= G|_{\underline{P}}$, $p_{s=1}\circ F$ is the constant homotopy in $t$, and $p_{t=0}\circ F$ as well as $p_{t=1}\circ F$ are constant homotopies in $s$. We want to lift $F$ to a map $\underline{P}\rightarrow \underline{A\otimes (s,ds)\otimes (t,dt)}$. We do this inductively over the $\mathcal{D}_n$ as above. Assume we have constructed the lift on $\underline{P}|_{\mathcal{D}_n}$. Again, let $U\in\mathcal{D}_L$ be a fixed subtorus of dimension $n+1$.
We consider the pullback square
\[\xymatrix{
 Q\ar[d]\ar[r]^{p_1} & \underline{A}(U,m_\mathcal{D}(U))\otimes S\ar[d]^{p_3}\\
\lim_{\mathcal{D}^+(U,m_\mathcal{D}(U))}\underline{A}\otimes (s,ds)\otimes(t,dt)\ar[r]^(.6){p_2} & \lim_{\mathcal{D}^+(U,m_\mathcal{D}(U))}\underline{A}\otimes S
}\]
and we want to find the dashed lift in the diagram
\[\xymatrix{
& \underline{A}(U,m_\mathcal{D}(U))\otimes (s,ds)\otimes(t,dt)\ar[d]^{p_4}\\
\ar@{-->}[ru]P(U,m_\mathcal{D}(U))\ar[r] & Q
}\]
where all solid arrows are induced by the given maps.
Doing this for any choice of $U$ then yields the desired extension to $\mathcal{D}_{n+1}$ by choosing the maps $\underline{P}(U,H)=\underline{P}(U,m_\mathcal{D}(U))\rightarrow \underline{A}(U,m_\mathcal{D}(U))\otimes (s,ds)\otimes (t,dt)\rightarrow \underline{A}(U,H)\otimes (s,ds)\otimes(t,dt)$ for any $(U,H)\in\mathcal{D}_{n+1}$. This will finish the induction. As for the existence of the lift, it suffices to show that $p_4$ is surjective on degree 2 cohomology. Then the existence of a lift up to homotopy follows from Lemma \ref{lem:cohomohomotopy} and this can be made strictly commutative by Lemma \ref{lem:homotopetocom} since $p_4$ is surjective. Note that $\underline{A}(U,m_\mathcal{D}(U))$ and $\lim_{\mathcal{D}^+(U,m_\mathcal{D}(U))}\underline{A}$ are cohomologically simply connected by Proposition \ref{prop:technicaldiagramstuff}. As $H^*(S)=H^*(S^1)$, we deduce that $p_2$ is an isomorphism on $H^2$. Furthermore, since $\underline{A}(U,m_\mathcal{D}(U))\rightarrow\lim_{\mathcal{D}^+(U,m_\mathcal{D}(U))}\underline{A}$ is an isomorphism on $H^0$, it follows that $p_3$ is an isomorphism on $H^1$. Using that $p_3$ is surjective, one deduces that $p_1$ is injective on $H^2$. But this implies that $p_4$ gives an isomorphism on $H^2$ since this holds for the composition $p_1\circ p_4$. This finishes the construction of the lift $\widetilde{F}\colon \underline{P}\rightarrow \underline{A\otimes (s,ds)\otimes (t,dt)}$.

Now consider the commutative diagram of solid arrows
\[\xymatrix{
\underline{P}\ar[d]\ar[r]^(.3){\widetilde{F}} & \underline{A\otimes (s,ds)\otimes(t,dt)}\ar[d]^{p_{s=0}}\\
\mathcal{M}\ar[r]^{H}\ar@{-->}[ur] & \underline{A\otimes (t,dt)}
}\]
and note that $p_{s=0}$ satisfies \textbf{(SCM)}. It follows from Propostion \ref{prop:lifting} that the dashed extension of $\widetilde{F}$ exists. By construction $p_{t=0}\circ \widetilde{F}$, $p_{s=1}\circ \widetilde{F}$, and $p_{t=1}\circ \widetilde{F}$ are homotopies relative to $\underline{P}$ which connect $f$ to $g$.
\end{proof}

Note that if $\underline{A}$ carries a cohomology $\underline{R}$-structure (or a $\underline{P}$-structure) and $\underline{A}\rightarrow\underline{B}$ is a morphism of systems, then $B$ inherits such a structure from $A$. Theorem \ref{thm:strictification}, Proposition \ref{prop:SCexistence}, and Remark \ref{rem:weakvscohomology} combine to give the following

\begin{cor}\label{cor:strictreplacement}
Let $\mathcal{D}\subset \mathcal{S}$ be a finite stable subset and $\underline{P}$ a free nilpotent $\mathcal{D}$-system satisfying $\underline{P}(U,H)=\underline{P}(U,T)$ for all $(U,H)\in \mathcal{D}$ together with a weak quasi isomorphism
$\underline{P}\rightarrow \underline{A_{PL}(*)}$. Let  $\underline{A}$ be a spacelike be $\mathcal{D}$-system with $H^1(\underline{A})=0$ and a cohomology $\underline{R}$-structure satisfying \textbf{(TC)}. Then there is a quasi isomorphism $\underline{A}\simeq \underline{A'}$ such that $\underline{A'}$ satisfies \textbf{(SC)} and a strict $\underline{P}$-structure on $\underline{A'}$ which induces the homology $\underline{R}$-structure coming from $\underline{A}$. Furthermore this $\underline{P}$-structure is unique up to homotopy of $\mathcal{D}$-systems.
\end{cor}

\section{Realization}
Before we get to the technical details of the realization process, let us give a brief overview of the main idea, which is in fact rather simple. Lemma \ref{lem:retract} below shows that gluing a free equivariant $k$-cell to a $T$-space $X$, changes $X_T$ (up to homotopy) by gluing in a nonequivariant $k$-cell. Thus as long as we glue free $T$-cells we have close control over what is happening on the Borel construction. Using this, Proposition \ref{prop:realization} shows that given a suitable cochain algebra $\mathcal{M}$ and a $T$-space $X$, we can modify $X$ by gluing free $T$-cells until $\mathcal{M}$ becomes a model for $X_T$. Of course in the end, we are not interested in only gluing free cells but rather want to apply the previous approximation process inductively to tori $T/H$, with $H$ running through all isotropy groups. This is done in Theorem \ref{thm:P-realization}, which achieves our goal of realizing whole systems of algebras.

\begin{lem}\label{lem:retract}
Let $X$ be a $T$-space and $i\colon X\rightarrow X_T$ be the inclusion $x\mapsto [e_0,x]$ for some $e_0\in ET$. For a map $S^n\rightarrow X$ denote by $\tilde{\varphi}\colon S^n\times T\rightarrow X$ the equivariant extension. Furthermore consider $D^{n+1}\cong D^{n+1}\times\{1\}\subset D^{n+1}\times T\subset (D^{n+1}\times T)_T$ as a subspace. Then the inclusion
\[ X_T\cup_{i\circ \varphi} D^{n+1}\longrightarrow  X_T\cup_{\tilde{\varphi}_T} (D^{n+1}\times T)_T \cong (X\cup_{\tilde{\varphi}} D^{n+1}\times T)_T\]
admits a deformation retraction.
\end{lem}

\begin{proof}
The map $(D^{n+1}\times T)_T=(ET\times D^{n+1}\times T)/T\rightarrow ET\times D^{n+1}$ given by $[x,v,t]\mapsto(xt,v)$ is a homeomorphism. Note that our choice of $ET$ (see \cite{Milnor}) is a CW-complex. The inclusion of the subcomplex $A:= (ET\times S^n)\cup (\{e_0\}\times D^{n+1})\rightarrow ET\times D^{n+1}$ is a homotopy equivalence since both spaces are contractible. Thus $ET\times D^{n+1}$ deformation retracts onto $A$ (e.g.\ by \cite[Lemma 4.6]{Hatcher}). This gives the desired deformation retraction $X_T\cup_{i\circ \varphi} D^{n+1}=X_T\cup_{\tilde{\varphi}_T}A\longleftarrow  X_T\cup_{\tilde{\varphi}_T} (D^{n+1}\times T)_T $, where $\tilde{\varphi}_T\colon (S^n\times T)_T\rightarrow X_T$ is the induced map on Borel constructions.
\end{proof}

The next step is to approximate an algebraic model by gluing free $T$-cells to a $T$-space. The key to this is the following

\begin{prop}\label{prop:realization}
Let $X$ be a finite type $T$-space and assume that either $X=\emptyset$ or that $\pi_2(X_T)\rightarrow \pi_2(BT)$ is surjective and $\pi_1(X)$ is finitely generated on every path component. Consider a diagram
\[\xymatrix{
\mathcal{M}\ar[r]^\phi & A_{PL}(X_T)\\
P\ar[u]\ar[r]& A_{PL}(BT)\ar[u]
}\]
in which $P\rightarrow\mathcal{M}$ is a disconnected Sullivan extension of cohomologically simply connected, finite type, disconnected almost Sullivan algebras. Assume further that $P\rightarrow A_{PL}(BT)$ is a quasi isomorphism and $H^2(P)\otimes H^0(\mathcal{M})\rightarrow H^2(\mathcal{M})$ is injective.
Then there is a componentwise simply connected $T$-space $Y$, which arises from $X$ by gluing free $T$-cells, such that $\phi$ factorizes as
\[\mathcal{M}\rightarrow A_{PL}(Y_T)\rightarrow A_{PL}(X_T)\]
where the first map is a quasi isomorphism and commutes with the $P$-actions. If $X$ is a $T$-CW-complex, then $Y$ can be chosen as a $T$-CW-complex. If the kernel and cokernel of $\phi^*\colon H^*(\mathcal{M})\rightarrow H^*(X_T)$ are finite dimensional, then one can choose $Y$ such that it arises from $X$ by gluing finitely many $T$-cells.
\end{prop}

Before we get to the proof, let us give a brief overview on the approximation step and sort out some technicalities. For an in depth treatment of the nonequivariant realization process of cochain algebras we refer e.g.\ to \cite[Chapter 17]{BibelI}, \cite[Section 1.6]{BibelII}. A collection of the needed results with more specific references was given in Remark \ref{rem:technicalstuff}. There are two pairs of functors
\[\{\text{cochain algebras}\}\leftrightarrows \{\text{simplicial sets}\}\leftrightarrows \{\text{topological spaces}\}.\]
The simplicial set of a cochain algebra $A$ is denoted $\langle A \rangle$. The space associated to a simplicial set $L$ is the Milnor realization and denoted $\vert L\vert$. We also write $\vert A\vert=\vert \langle A\rangle\vert$ for the spatial realization of a cochain algebra. The simplicial set associated to a space is the set $Sing(X)$ of its singular simplices. The cochain algebra of a simplicial set is denoted $A_{PL}(L)$. Recall that $A_{PL}(X)=A_{PL}(Sing(X))$ by definition. 

The canonical map $\vert Sing(X)\vert\rightarrow X$ is always a weak equivalence. On the side of algebras, there is a similar phenomenon: the functors $\langle -\rangle$ and $A_{PL}(-)$ are adjoint to one another and there is the canonical morphism $A\rightarrow A_{PL}(\langle A\rangle)$ which is adjoint to the identity on $\langle A\rangle$. If $A$ is a disconnected almost Sullivan algebra of finite type and $H^1(A)=0$, then this is a quasi isomorphism (cf.\ Remark \ref{rem:technicalstuff}). In this case, also the map $A_{PL}(\vert A\vert)\rightarrow A_{PL}(\langle A\rangle)$ induced by the inclusion $\langle A\rangle\rightarrow Sing(\vert A\vert)$ is a surjective quasi isomorphism, such that $A$ is in fact a model for its realization $\vert A\vert$.

Now in the situation of Proposition \ref{prop:realization}, the cdga morphisms $\mathcal{M}\rightarrow A_{PL}(X_T)$, $P\rightarrow A_{PL}(BT)$ are adjoint to morphisms $\langle \mathcal{M}\rangle\leftarrow Sing(X_T)$, $\langle P\rangle\leftarrow Sing(BT)$ and we obtain an induced commutative diagram

\[\xymatrix{\vert \mathcal{M}\vert \ar[d]&\ar[d]\ar[l]_f \vert Sing(X_T)\vert\ar[r]^g &\ar[d] X_T\\
\vert P \vert & \vert Sing (BT)\vert\ar[l]\ar[r] & BT
}\]
of spaces. The right hand horizontal maps are the canonical weak equivalences and the left hand vertical map is a fibration (cf.\ Remark \ref{rem:technicalstuff}). We want $f$ to become a weak equivalence. This can be achieved as in standard CW-approximation by gluing (nonequivariant) cells to $\vert Sing(X_T)\vert$ and $X_T$ accordingly. To make this precise, we consider a section $s$ of $\vert Sing(S^{k-1})\vert\rightarrow S^{k-1}$ induced by a triangulation of $S^{k-1}$.
Now given some $\varphi\colon S^{k-1}\rightarrow X_T$ let $\varphi':=s\circ\varphi\colon S^{k-1}\rightarrow \vert Sing(X_T)\vert$ the corresponding lift. We have $\varphi=g\circ\varphi'$. Then for some extension $\psi\colon D^k\rightarrow\vert \mathcal{M}\vert$ of $f\circ \varphi'$ we obtain continuous maps
\[\vert\mathcal{M}\vert\leftarrow \vert Sing(X_T)\vert\cup_{\varphi'} D^k \rightarrow X_T\cup_\varphi D^k\]
where the right hand map is still a weak equivalence and the left hand map might be closer to a weak equivalence than before, depending on the choice of $\varphi,\psi$. This is essentially the same as taking a homotopy inverse of $g$ (in case $X_T$ is a simply connected CW-complex) and applying usual CW-approximation to the resulting composition $\vert \mathcal{M}\vert \leftarrow X_T$. However doing things as above makes it easier to keep track of strict commutativity of diagrams.

Now the idea is to go through this nonequivariant approximation procedure for the Borel constructions while gluing free $T$-cells to $X$.
This is done as follows: fix a base point in $ET$ and identify $X\subset X_T$ with the fiber at that point. Now with $\varphi,\psi$ as above, assume additionally that $\varphi$ takes image in $X$ and $\psi$ maps $D^k$ to a single fibre of $\vert \mathcal{M}\vert\rightarrow \vert P\vert$. The assumption on $\psi$ makes sense since in this case the image of $\varphi'$ maps to a single point in $\vert Sing(BT)\vert$. Denote by $Y$ the $T$-space obtained from $X$ by attaching a free $D^k\times T$-cell along the equivariant extension of $\varphi$. Recall from Lemma \ref{lem:retract} that $X_T\cup_\varphi D^k\subset Y_T$ is a deformation retract. Then the following technical Lemma tells us that the above approximation step may be carried out while maintaining the setup of Proposition \ref{prop:realization}.

\begin{lem}\label{lem:diagramirrsinn}
With the setup above, there is a map $\mathcal{M}\rightarrow A_{PL}(Y_T)$ fitting into the commutative diagram
\[\xymatrix{ \mathcal{M}\ar[r] & A_{PL}(Y_T)\ar[d]\ar[r]& A_{PL}(X_T)\ar[dl] \\ P\ar[u]\ar[r] & A_{PL}(BT)& }\]
as well as into the homotopy commutative diagram
\[\xymatrix{
\mathcal{M}\ar[d] \ar[r] &  A_{PL}(Y_T)\ar[d]\\
A_{PL}(\vert\mathcal{M}\vert)\ar[r] & A_{PL}(\vert Sing(X_T)\vert\cup_{\varphi'} D^k)
}\]
in which the vertical maps are quasi isomorphisms.

\end{lem}

\begin{proof}
For a space $Z$, a morphism of cochain algebras $\varphi\colon A\rightarrow A_{PL}(Z)$ induces a commutative diagram
\[\xymatrix{
A \ar[r]\ar[dr] & A_{PL}(\langle A\rangle)\ar[d] &\ar[l] A_{PL}(\vert A\vert)\ar[d]\\ & A_{PL}( Z)& \ar[l] A_{PL}(\vert Sing(Z)\vert)
}\]
of cochain algebras. Here the horizontal maps in the square are induced by the canonical inclusions of simplicial sets which, for some simplicial set $L$, is the morphism $L\rightarrow Sing(\vert L\vert)$ adjoint to $\id_{\vert L\vert}$. The vertical maps in the square are induced by the simplicial morphism $Sing(Z)\rightarrow \langle A\rangle$ adjoint to $\varphi$. In the case of $A=\mathcal{M},P$ we may lift $A\rightarrow A_{PL}(\langle A\rangle)$ through the surjective quasi isomorphism $A_{PL}(\langle A\rangle)\leftarrow A_{PL}(\vert A\vert)$. Doing this for $P$ first and then lifting to $A_{PL}(\vert\mathcal{M}\vert)$ relative to $P$, we obtain the commutative diagram
\[\xymatrix{
& A_{PL}(\vert Sing(X_T)\vert \cup_{\varphi'} D^k)\ar[dr] & A_{PL}(\vert Sing(X_T\cup_\varphi D^k)\vert)\ar[l]\ar[d]\ar[r] & A_{PL}(X_T\cup_\varphi D^k)\ar[d]\\
\mathcal{M}\ar[r]& A_{PL}(\vert\mathcal{M}\vert)\ar[u]\ar[r]& A_{PL}(\vert Sing(X_T)\vert)\ar[r]& A_{PL}(X_T)\\
P\ar[u]\ar[r]& A_{PL}(\vert P\vert)\ar[r]\ar[u]& A_{PL}(\vert Sing(BT)\vert)\ar[r]\ar[u]\ar@/_{2.0cm}/[uu]& A_{PL}(BT)\ar[u]\ar@/_{2.0cm}/[uu]
}\]
in which the compositions of the bottom rows are the given maps $\mathcal{M}\rightarrow A_{PL}(X_T)$ and $P\rightarrow A_{PL}(BT)$. All morphisms within and between the two central columns as well as within the right hand column are induced by respective continuous maps. Of those, only $X_T\cup D^k\rightarrow BT$ and $\vert Sing(X_T)\vert \cup \varphi' D^k\rightarrow \vert Sing(X\cup_\varphi D^k)\vert$ need further specification. The first map just extends $X_T\rightarrow BT$ by mapping $D^k$ to the basepoint. The second map is the obvious inclusion on $\vert Sing(X)\vert$ and on $D^k$ it is induced by a fixed choice of triangulation of $D^k$ extending the previous choice of triangulation of $S^{k-1}$. Now commutativity of the diagram can be checked on the level of continuous maps, using that the discs in the top row map to single points in the spaces corresponding to the bottom row.

Every algebra in the above diagram comes with a compatible morphism from $P$ and the morphisms in the top row are quasi isomorphisms. Thus we can lift up to homotopy relative to $P$ to obtain a morphism $\mathcal{M}\rightarrow A_{PL}(X_T\cup_\varphi D^k)$ and we can further lift this  through $A_{PL}(X_T\cup_\varphi D^k)\leftarrow A_{PL}(Y_T)$ up to homotopy relative to $P$. The map $A_{PL}(Y_T)\rightarrow A_{PL}(X_T)$ is surjective so the lemma follows from Lemma \ref{lem:homotopetocom}.
\end{proof}

\begin{proof}[Proof of Proposition \ref{prop:realization}]

We first treat the case when $X\neq \emptyset$. Furthermore, we begin under the assumption that $X_T$ is path connected, simply connected and $\mathcal{M}$ is connected. These assumptions will be justified in the end in a fashion which is analogous to the following induction step.

Assume inductively that we have a $T$-space $Y\supset X$ and a factorization $\mathcal{M}\rightarrow A_{PL}(Y_T)\rightarrow A_{PL}(X_T)$ of $P$-cdgas such that for some $k\geq 2$ the map $H^*(\mathcal{M})\rightarrow H^*(Y_T)$ is an isomorphism in degrees $<k-1$ and injective in degree $k-1$.

For a space $Z$ set $\tilde{Z}:=\vert Sing(Z)\vert$. Furthermore, whenever a choice of map $\phi\colon Z\rightarrow Z'$ is clear from the context, we denote by $H_*(Z',Z)$ (resp.\ $\pi_*(Z',Z)$) the relative homology (resp.\ homotopy) groups of $(M_\phi, Z)$, where $M_\varphi$ denotes the mapping cylinder of $\varphi$. This construction is clearly functorial with respect to commutative squares of continuous maps. Consider the diagram

\[\xymatrix{
F\ar[d] & \tilde{Y}\ar[d]\ar[l]\ar[r] & Y\ar[d]\\
\vert \mathcal{M}\vert\ar[d] & \tilde{Y}_T\ar[d]\ar[l]_f\ar[r] & Y_T\ar[d]\\
\vert P\vert & \tilde{BT}\ar[l]\ar[r]& BT
}\]
where $F$ denotes the fiber of $\vert \mathcal{M}\vert\rightarrow\vert P\vert$. Dualizing the assumptions we obtain that on homology, $f_*\colon H_*(\tilde{Y}_T)\rightarrow  H_*(\vert\mathcal{M}\vert)$ is an isomorphism in degrees $<k-1$ and surjective in degree $k$. In other words, $H_i(\vert \mathcal{M}\vert,\tilde{Y}_T)=0$ for $i<k$.  Choose some nontrivial $x\in H_k(\vert \mathcal{M}\vert,\tilde{Y}_T)$. We want to glue a free $T$-cell to $Y$ and extend $f$ such that $x$ becomes trivial.

The rational Hurewicz theorem provides an isomorphism $H_k(\vert \mathcal{M}\vert,\tilde{Y}_T)\cong\pi_k(\vert\mathcal{M}\vert,\tilde{Y}_T)\otimes \mathbb{Q}$. The first step is to show that $x$ comes from an element in $\pi_k(F,\tilde{Y})\otimes \mathbb{Q}$. We do this by showing surjectivity of the arrow marked as $\alpha$ in the diagram

\[\xymatrix{
\pi_k(\tilde{Y})\ar[d]\ar[r] &\pi_k(F)\ar[d]\ar[r]& \pi_k(F,\tilde{Y})\ar[d]^{\alpha}\ar[r] &\pi_{k-1}(\tilde{Y})\ar[d]\ar[r]&\pi_{k-1}(F)\ar[d]^{i}\\
\pi_k(\tilde{Y}_T)\ar[d]^{p}\ar[r]& \pi_k(\vert \mathcal{M}\vert)\ar[d]\ar[r]& \pi_k(\vert \mathcal{M}\vert,\tilde{Y}_T)\ar[d]\ar[r] &\pi_{k-1}(\tilde{Y}_T)\ar[d]\ar[r]&\pi_{k-1}(\vert \mathcal{M}\vert)\ar[d]\\
\pi_k(\tilde{BT})\ar[r]& \pi_k(\vert {P}\vert)\ar[r]& \pi_k(\vert {P}\vert,\tilde{BT})\ar[r] &\pi_{k-1}(\tilde{BT})\ar[r]&\pi_{k-1}(\vert {P}\vert)\\
}\]
after tensoring everything with $\mathbb{Q}$. Note that for $k=2$ the second to right column becomes zero and we only need to use the left hand part of the diagram, where consequently all groups are Abelian. The proof is a straight forward diagram chase using the following facts: rows are exact, all columns except the central one are known to be exact, $\pi_k(\vert \mathcal{P}\vert,\tilde{BT})\otimes \mathbb{Q}=0$, $i$ is injective, and $p$ is surjective. Surjectivity of $p$ holds by assumption and $\pi_k(\vert \mathcal{P}\vert,\tilde{BT})\otimes \mathbb{Q}=0$ holds since $\tilde{BT}\rightarrow \vert P\vert$ is a rational equivalence. Finally, injectivity of $i$ is equivalent to surjectivity of $\pi_{k}(\vert M\vert)\otimes \mathbb{Q}\rightarrow \pi_{k}(\vert P\vert)\otimes \mathbb{Q}$. This is only an issue for $k=2$ where the injectivity of $i$  is in fact not needed for the argument, as explained above.
 
We have shown that $x\in H_k(\vert \mathcal{M}\vert,\tilde{Y}_T)$ comes from an element $y\in \pi_k(F,\tilde{Y})\otimes \mathbb{Q}$. Possibly modifying $x$ and $y$ up to nonzero scalars, we choose a representative $\phi\colon (D^k,S^{k-1})\rightarrow (M_{\overline{f}}, \tilde{Y})$ for $y$, where $\overline{f}\colon \tilde{Y}\rightarrow F$ is the restriction of $f$. In doing this, we may assume that $\phi|_{S^{k-1}}$ is of the form $\varphi'$ for some $\varphi\colon S^{k-1}\rightarrow Y$ with notation as above Lemma \ref{lem:diagramirrsinn}. Now we glue an equivariant cell $D^k\times T$ to $Y$ along the equivariant extension of $\varphi$, resulting in the $T$-space $Y'$. Furthermore set $\psi$ to be the composition $D^k\rightarrow M_{\overline{f}}\rightarrow F$. It extends the map $f\circ\varphi'$ and defines an extension of $f$ to $\tilde{Y}_T\cup_{\varphi'} D^k$.
By Lemma \ref{lem:diagramirrsinn} we have a factorization $\mathcal{M}\rightarrow A_{PL}(Y'_T)\rightarrow A_{PL}(Y_T)$ of $P$-cdgas such that on cohomology the first map is given by
\[H^*(\mathcal{M})\cong H^*(\vert \mathcal{M}\vert)\xrightarrow{f^*} H^*(\tilde{Y}_T\cup_{\varphi'} D^k)\cong H^*(Y_T').\]
We claim that the total dimension of the kernel and cokernel of this extended $f^*$ has been reduced by $1$. Indeed, one checks that by construction, the map $H_*(\vert\mathcal{M}\vert, \tilde{Y}_T)\rightarrow H_*(\vert \mathcal{M}\vert, \tilde{Y}_T\cup_{\varphi'} D^k)$ is an isomorphism in all degrees except in degree $k$, where \[H_k(\vert \mathcal{M}\vert, \tilde{Y}_T\cup_{\varphi'} D^k)= H_k(\vert \mathcal{M}\vert, \tilde{Y}_T)/\langle x\rangle_\mathbb{Q}.\] By iterating this procedure we may assume $H_k(\vert \mathcal{M}\vert, \tilde{Y}'_T)=0$, which finishes the induction.

To justify the initial assumption of path connectedness and simply connectedness, we proceed as in the induction with few modifications. Decompose $\mathcal{M}=\prod_i \mathcal{M}_i$ into path components. For a component $\mathcal{M}_i$ which maps trivially to $A_{PL}(X)$, we choose a simply connected realization $Y_i$ and a quasi isomorphism $\mathcal{M}_i\rightarrow A_{PL}((Y_i)_T)$ as in the $X=\emptyset$ case described below. Then we replace $X$ with $X\sqcup Y_i$ and extend $\mathcal{M}\rightarrow A_{PL}(X\sqcup Y_i)=A_{PL}(X)\times A_{PL}(Y_i)$ accordingly. Thus we may assume $H^0(\mathcal{M})\rightarrow H^0(X_T)$ to be injective.
Then we glue free $1$-cells to $X$ in order to obtain a $T$-space $Y\supset X$ and $\mathcal{M}\rightarrow A_{PL}(Y_T)$ which induces an isomorphism on $H^0$. This is done as in the induction above by killing $H_1(\vert \mathcal{M}\vert, \tilde{Y}_T)$ (note that for $k=1$, the surjectivity of $\alpha$ can be shown directly and tensoring with $\mathbb{Q}$ is neither needed nor sensical).

Having an isomorphism on $H^0$, we may consider path components separately and thus assume $\mathcal{M}$ and $X$ to be path connected.
We continue to make spaces simply connected. This is done as before, the difference being that we do not start with an element in homology which needs to be killed but directly with a nontrivial element $x\in\pi_1(Y_T)$.
By assumption, the composition $\pi_2(X_T)\rightarrow \pi_2(Y_T)\rightarrow \pi_2(BT)$ is surjective. Hence the second map is surjective as well proving that $\pi_1(Y)\rightarrow \pi_1(Y_T)$ is an isomorphism. Thus we may choose a representative $\varphi\colon S^1\rightarrow Y$.
We claim that $\pi_1(F)=0$ and hence the composition $f\circ\varphi'$ is nullhomotopic. To see this, we use the fact that that $\vert\mathcal{M}\vert\rightarrow\vert P\vert$ induces an injection on $H^2$, or dually a surjection in $H_2$ with rational coefficients. But the spaces are simply connected and rational so $\pi_2(\vert\mathcal{M}\vert)\rightarrow\pi_2(\vert P\vert)$ is surjective and thus indeed $\pi_1(F)=0$. It follows that $f\circ\varphi'$ extends to a map $\psi\colon D^2\rightarrow F$ and we proceed as in the induction by gluing a free $2$-cell along $\varphi$. Repeating this we can achieve that $\pi_1(Y_T)=\pi_1(Y)=0$.

We now consider the case $X=\emptyset$ and assume $\mathcal{M}$ to be cohomologically connected by considering path components separately. We essentially use the procedure for realizing algebraic models by free actions as described e.g.\ in \cite[Proposition 4.2]{HalperinTori}. The need for a slightly different argument is due to the fact that the approximation in the $X\neq \emptyset$ case relies on the surjectivity of $\pi_2(X_T)\rightarrow \pi_2(BT)$ in order for the resulting space to be simply connected.

Start by taking a nonequivariant rational CW approximation $Z\rightarrow\vert \mathcal{M}\vert$, i.e.\ a rational equivalence where $Z$ is a simply connected CW-complex.
Note that -- as in the more complicated, equivariant argument above -- we may take $Z$ to be a finite CW-complex if $H^*(\mathcal{M})$ is finite dimensional. Let $X_1,\ldots,X_r$ be cohomology classes in $H^*(P)\cong H^*(BT;\mathbb{Q})$ which correspond to a basis of the integral cohomology $H^2(BT;\mathbb{Z})$. Using that $H^2(P)\rightarrow H^2(\vert M\vert)$ is injective, we see that $Z\rightarrow\vert \mathcal{M}\vert$ can be chosen in a way such that the images of the $X_i$ in $H^2(Z,\mathbb{Q})$ come from integral cohomology classes which are part of a basis of $H^2(Z;\mathbb{Z})$. Now use these classes to define a map $f\colon Z\rightarrow K(\mathbb{Z}^r,2)=BT$ and pull back the universal $T$-bundle $T\rightarrow ET\rightarrow BT$ along $f$. The result is a principal $T$-bundle $T\rightarrow Y\rightarrow Z$.
Using the Hurewicz theorem and dualizing we obtain that the map $\pi_2(Z)\rightarrow \pi_2(BT)$ is surjective. An argument using the naturality of the long exact homotopy sequence then shows that $Y$ is simply connected. Note that the (finite) CW-structure on $Z$ translates cellwise to a (finite) $T$-CW-structure on $Y$. It remains to show that $P\rightarrow \mathcal{M}$ is a model for $Y_T\rightarrow BT$. For this we refer to \cite[Proposition, 7.17]{AlgebraicModels}, \cite[Proposition 2.7]{AmannZoller}.

Finally note that if, by abuse of notation, $Y$ now denotes result of the procedure above, then $A_{PL}(Y_T)$ is the limit over the $A_{PL}(-)$ in the intermediate steps. Thus we obtain the desired factorization $\mathcal{M}\rightarrow A_{PL}(Y_T)\rightarrow A_{PL}(X_T)$.
Clearly, if the kernel and cokernel of the initial map $H^*(\mathcal{M})\rightarrow H^*(X_T)$ are finite dimensional, then the process is finished after gluing finitely many cells.
\end{proof}

\begin{thm}\label{thm:P-realization}
Let $C$ be a bounded collection of subgroups of $T$ and $\mathcal{D}\subset \mathcal{S}$ be the stable subset generated by $C$. Let $\underline{P}\rightarrow \underline{A_{PL}(*)}$ be an almost Sullivan model and let $\underline{P}\rightarrow\underline{A}$ be a $\mathcal{D}$-system under $\underline{P}$. Then there is a $T$-simply connected, $T$-finite $\mathbb{Q}$-type $T$-CW-complex $X$ with isotropies contained in $C$ such that $\underline{A_{PL}(X)}$ is connected to $\underline{A}$ via quasi isomorphisms of systems under $\underline{P}$ if and only if $\underline{A}$ satisfies \textbf{(TC)}, is spacelike, of finite type, $H^1(\underline{A})=0$, and $H^2(\underline{P})\otimes H^0(\underline{A})\rightarrow H^2(\underline{A})$ is injective. If $C$ is finite, then $X$ can be chosen as a finite $T$-CW-complex if and only if $\underline{A}$ satisfies \textbf{(LC)} and $H^*(\underline{A}(U,H))$ is finitely generated as an $H^*(\underline{P}(U,H))$-module for any $(U,H)\in\mathcal{D}$.
\end{thm}

\begin{proof} For a $T$-simply connected, $T$-finite $\mathbb{Q}$-type, $T$-space $X$, the system $\underline{A_{PL}(X)}$ has the described algebraic properties (cf.\ Remark \ref{rem:TC}). Conversely, let $\underline{A}$ be such a system under $\underline{P}$.
Let $\mathcal{D}_n$ consist of all $(U,H)\in \mathcal{D}$ such that the length of any maximal chain $(H_0,H)<\ldots<(H_0,T)$ is $\leq n$. Suppose inductively that we have a $T$-space $X$ such that each occurring isotropy group $H$ satisfies $(H_0,H)\in\mathcal{D}_n$, $X^H$ is componentwise {simply connected} for any $(H_0,H)\in \mathcal{D}_n$, and there is a chain of quasi isomorphisms \[\underline{A}|_{\mathcal{D}_n}\leftarrow \underline{B_1} \rightarrow\ldots\rightarrow \underline{B_k}\leftarrow \underline{A_{PL}(X)}|_{\mathcal{D}_n}\] of systems under $\underline{P}|_{\mathcal{D}_n}$. By  Proposition \ref{prop:SCexistence} we may assume that $\underline{A}$ and the $\underline{B}_i$ satisfy \textbf{(SC)}.

Now fix $(H_0,H)\in \mathcal{D}_{n+1}\backslash \mathcal{D}_n$. Define subsets $\mathcal{C}(H):=\mathcal{D}^-(H_0,H)=\mathcal{D}_n\cap \mathcal{D}(H_0,H)$, $\mathcal{C}(H)^+:= \mathcal{C}(H)\cup \{ (H_0,H)\}$, and $\tilde{\mathcal{C}}(H):=\mathcal{D}(H_0,H)\cup\{(H_0,H)\}$. We construct a space $Y(H)$ by gluing $T/H$-cells to $X^H$, as well as a chain of quasi isomorphisms linking $\underline{A}|_{\tilde{\mathcal{C}}(H)}$ to $\underline{A_{PL}}(Y(H))|_{\tilde{\mathcal{C}}(H)}$. Doing this for all choices of $H$ and piecing together the chains of quasi isomorphisms will then complete the induction step.

We start by extending the $\underline{B_i}|_{\mathcal{C}(H)}$-systems to $\mathcal{C}^+(H)$-systems. Consider the decomposition $\underline{A}(H_0,H)=\prod_i \underline{A}(H_0,H)_i$ into cohomologically connected cochain algebras. Set $P(H_0)=\underline{P}(H_0,H)$ and let $(P(H_0)\otimes \Lambda V_i,d)\rightarrow \underline{A}(H_0,H)$ be a quasi isomorphism from a Sullivan extension, compatible with the $P(H_0)$-actions. We may write $P(H_0)$ as a tensor product of $R_{T/H}$ and a contractible cdga and construct the extensions such that $d(V_i)\subset R_{T/H}\otimes \Lambda V_i$. Then $(P(H_0)\otimes \Lambda V_i,d)$ is clearly an almost Sullivan algebra. Set $\mathcal{M}(H_0,H)=\prod_i (P(H_0)\otimes\Lambda V_i,d)$. The map $\lim_{\mathcal{C}(H)}\underline{B_1}\rightarrow \lim_{\mathcal{C}(H)}\underline{A}$ is a quasi isomorphism by Proposition \ref{prop:SC}. Thus we may lift $\mathcal{M}(H_0,H)$ through this quasi isomorphism up to homotopy relative to $P(H_0)$. This yields a strictly commutative diagram

\[\xymatrix{
\underline{A}(H_0,H)\ar[d] & & \mathcal{M}(H_0,H)\ar[ll]\ar[d]\ar[dr]\ar[dl] & \\
\lim_{\mathcal{C}(H)}\underline{A}& \lim_{\mathcal{C}(H)}\underline{A}\otimes (t,dt)\ar[l] \ar[r]& \lim_{\mathcal{C}(H)}\underline{A} &\lim_{\mathcal{C}(H)}\underline{B_1}\ar[l]
}\]
which is compatible with the canonical $P(H_0)$-actions. Setting the $(H_0,H)$-positions to be $\mathcal{M}(H_0,H)$ we obtain extensions of $\underline{A}$, $\underline{A\otimes (t,dt)}$, and $\underline{B_1}$ to $\mathcal{C}^+(H)$-diagrams, as well as quasi isomorphisms between these extensions. Using the fact that $\lim_{\mathcal{C}(H)}\underline{B_{i-1}}\leftarrow\lim_{\mathcal{C}(H)}\underline{B_{i}}$ are quasi-isomorphisms as well, we may continue this procedure to form a chain of quasi isomorphisms of extended $\mathcal{C}(H)^+$-diagrams which are the identity in the $(H_0,H)$-position.

We claim that $\lim_{\mathcal{C}(H)}\underline{B_k}\leftarrow \lim_{\mathcal{C}(H)}\underline{A_{PL}(X)}$
is a quasi isomorphism as well. Let $\mathcal{D}^-(H)\subset \mathcal{C}(H)$ denote the category given by all $(H_0,K)\in\mathcal{D}$ with $H\subsetneq K$. Then the projections $\lim_{\mathcal{C}(H)}\underline{B_k}\rightarrow \lim_{\mathcal{D}^-(H)}\underline{B_k}$ and $\lim_{\mathcal{C}(H)}\underline{A_{PL}(X)}\rightarrow \lim_{\mathcal{D}^-(H)}\underline{A_{PL}(X)}$ are isomorphisms. By Lemma \ref{lem:subcomplexes}, $\underline{A_{PL}(X)}|_{\mathcal{D}^-(H)}$ satisfies \textbf{(SC)}. The system $\underline{B_k}|_{\mathcal{D}^-(H)}$ satisfies \textbf{(SC)} as well since $\underline{B_k}(H_0,K)\rightarrow \lim_{\mathcal{D}^-(H_0,K)} \underline{B_k}\cong \lim_{(\mathcal{D}^-(H))(H_0,K)} \underline{B_k}$ is surjective for any $(H_0,K)\in \mathcal{C}(H)$ by Proposition \ref{prop:technicaldiagramstuff}.
Thus $\lim_{\mathcal{D}^-(H)}\underline{B_k}\leftarrow \lim_{\mathcal{D}^-(H)}\underline{A_{PL}(X)}$ is a quasi isomorphism by Proposition \ref{prop:SC}.

With the claim proved, it follows that we obtain a chain of quasi isomorphisms of $\mathcal{C}^+(H)$-systems linking $\underline{A}|_{\mathcal{C}^+(H)}$ to $\underline{A_{PL}(X)}^+$ where the latter is defined to be $\underline{A_{PL}(X)}$ on $\mathcal{C}(H)$ and $\mathcal{M}(H_0,H)$ at $(H_0,H)$. Also by Lemma \ref{lem:subcomplexes}, the map $A_{PL}(X_{T/H_0}^H)\rightarrow \lim_{\mathcal{D}^-(H)}\underline{A_{PL}(X)}$ is a surjective quasi isomorphism. Thus by relative lifting we obtain a map $\mathcal{M}(H_0,H)\rightarrow A_{PL}(X_{T/H_0}^H)$ which is compatible with the $P(H_0)$-action.

Note that either $X^H=\emptyset$ or there is some $H'\supsetneq H$ for which $X^{H'}$ is nonempty and componentwise simply connected. In the latter case, $\pi_2(X^{H'})\rightarrow \pi_2(X^H)\rightarrow \pi_2(BT)$ is surjective on each component, which also implies surjectivity of the second map in the composition. In this case furthermore $\pi_1(X^H)$ is finitely generated. We can now apply Proposition \ref{prop:realization} to obtain a simply connected $T$-CW-complex $Y(H)$ which arises from $X$ by gluing $T/H$-cells and a factorization
\[\xymatrix{
\mathcal{M}(H_0,H)\ar[d]\ar[r]\ar[rd] & A_{PL}(Y(H)_{T/H_0})\ar[d]\\
\lim_{\mathcal{D}^-(H)}\underline{A_{PL}(X)} & \ar[l] A_{PL} (X^H_{T/H_0})
}\]
in which the top horizontal map is a quasi isomorphism of $P(H_0)$-cdgas.
Thus $\underline{A}|_{\mathcal{C}^+(H)}$ and $\underline{A_{PL}(Y(H))}|_{\mathcal{C}^+(H)}$ are linked by a chain of quasi isomorphisms compatible with the $(\underline{P}|_{\mathcal{C}(H)^+}$)-actions.

We extend this to a chain of quasi isomorphisms of $\tilde{\mathcal{C}}(H)$-systems. The extensions of the starting and ending systems $\underline{A}$ and $\underline{A_{PL}(Y)}$ are predetermined. For the systems in between recall that the $(H_0,H)$ position is always $\mathcal{M}(H_0,H)$. For those systems we set the $(U,H)$-positions, $U\subset H_0$, to be $\mathcal{M}(U,H):=P(U,H)\otimes_{P(H_0)} \mathcal{M}(H_0,H)=\prod_i(P(U,H)\otimes \Lambda V_i,d)$ with the obvious differentials and maps between them. Note that the quasi isomorphisms of systems extend naturally as well, providing us with the desired chain of $\tilde{\mathcal{C}}(H)$-system morphisms linking $\underline{A}|_{\mathcal{C}^+(H)}$ to $\underline{A_{PL}(Y(H)}|_{\mathcal{C}^+(H)}$. It remains to see that these are quasi isomorphisms. Since we have quasi isomorphisms at the $(H_0,H)$ positions, this follows from the fact that all systems satisfy \textbf{(TC)}. This holds by assumption at the start and the end of the chain and follows from Lemma \ref{lem:tensorTC} for the intermediate systems.

Having done this for every choice of $H$ we define $Y=\bigcup_{H\in \mathcal{D}_{n+1}\backslash \mathcal{D}_n} Y(H)$. For $H_1,H_2\in \mathcal{D}_{n+1}\backslash \mathcal{D}_n$, the quasi isomorphisms chains between the $\underline{A}|_{\tilde{\mathcal{C}}(H_i)}$ and $\underline{A_{PL}(Y)}|_{\tilde{\mathcal{C}}(H_i)}$ agree when restricted to
${\tilde{\mathcal{C}}(H_1)}\cap {\tilde{\mathcal{C}}(H_2)}$. Thus they all piece together to form a quasi isomorphism chain from $\underline{A}|_{\mathcal{D}_{n+1}}$ to $\underline{A_{PL}(Y)}|_{\mathcal{D}_{n+1}}$. This finishes the induction.

Finally, assume that $C$ is finite, $\underline{A}$ satisfies \textbf{(LC)}, and the modules $\underline{R}\rightarrow H^*(\underline{A})$ are pointwise finitely generated. We claim that in the construction above, the kernel and cokernel of $H^*(\mathcal{M}(H_0,H))\rightarrow H^*(X^H_{T/H_0})$ are finite dimensional and one can therefore choose $Y(H)$ as in Proposition \ref{prop:realization} such that it arises from $X^H$ by gluing a finite number of cells. To verify the claim let $U\supsetneq H$ be a torus. When localizing the commutative diagram

\[\xymatrix{H^*(\mathcal{M}(H_0,H))\ar[r]\ar[dr]& H^*(X^H_{T/H_0})\ar[d]\\
& H^*(X^U_{T/H_0})
}\]
at $S_{T/H}(U)$,
the diagonal map is an isomorphism by \textbf{(LC)} as $H^*(X^U_{T/H_0})\cong H^*(\underline{A}(H_0,m_\mathcal{D}(U)))$ (see Lemma \ref{lem:isotropytypes}). The vertical map is an isomorphism by the Borel Localization theorem (cf. Remark \ref{rem:localization}). In particular the horizontal map becomes an isomorphism after localization. Furthermore \cite[Proposition 3.10.1]{AP} implies that $H^*(X^H_{T/H_0})$ is finitely generated as an $R_{T/H_0}$ module and the claim follows from Lemma \ref{lem:localizationfinite}. Conversely, if $X$ is a finite $T$-CW-complex, then \cite[Proposition 3.10.1]{AP} as well as Borel localization ensure that $H^*(\underline{A_{PL}(X)})$ is finitely generated over $\underline{R}$ and satisfies \textbf{(LC)}.
\end{proof}

Combining Theorem \ref{thm:P-realization} with Corollary \ref{cor:strictreplacement} now yields

\begin{thm}\label{thm:R-realization}
Let $C$ be a finite collection of subgroups of $T$ and $\mathcal{D}\subset \mathcal{S}$ be the stable subset generated by $C$. Let $\underline{A}$ be a $\mathcal{D}$-system with a cohomology $\underline{R}$-structure. Then there is a $T$-simply connected, $T$-finite $\mathbb{Q}$-type $T$-CW-complex $X$ with isotropies contained in $C$ such that $\underline{A_{PL}(X)}$ is connected to $\underline{A}$ via quasi isomorphisms of systems respecting the cohomology $\underline{R}$-structures if and only if $\underline{A}$ satisfies \textbf{(TC)}, is spacelike, of finite type, $H^1(\underline{A})=0$, and $\underline{R}^2\otimes H^0(\underline{A})\rightarrow H^2(\underline{A})$ is injective. The space $X$ can be chosen as a finite $T$-CW-complex if and only if $\underline{A}$ satisfies \textbf{(LC)} and $H^*(\underline{A}(U,H))$ is finitely generated as an $R_{T/U}$-module for any $(U,H)\in\mathcal{D}$.
\end{thm}

\section{Applications}
\subsection{Realization of equivariant cohomology algebras}
We want to apply Theorem \ref{thm:R-realization} to derive an algebraic criterion of when a given algebra is the equivariant cohomology algebra of a finite $T$-CW-complex. Recall that $R_T:=H^*(BT)\cong \mathbb{Q}[V]$ for some finite dimensional vector space $V$ in degree $2$. We aim to give a formulation which is devoid of the internal terminology of this article and which uses a simplified version of the localization condition.

\begin{thm}\label{thm:cohorealization}
Let $A$ be a graded $\mathbb{Q}[V]$-algebra. Then $A$ is isomorphic to the equivariant cohomology algebra of a $T$-simply connected finite $T$-CW-complex if and only if there is a finite collection $V_0,\ldots, V_k\subset V$ of vector spaces, $\mathbb{Q}[V_i]$-algebras $A_i$ for $i=0,\ldots,k$, and for each inclusion $V_i\supset V_j$ a map $f_{ij}\colon A_i\rightarrow \mathbb{Q}[V_i] \otimes_{\mathbb{Q}[V_j]} A_j$ of $\mathbb{Q}[V_i]$-algebras satisfying the following properties:
\begin{enumerate}[(i)]
\item  $V_0=V$, $V_k=0$, and for any $i,j$ we have $V_i+ V_j=V_l$ for some $0\leq l\leq k$.
\item $A_0=A$, each $A_i$ is spacelike, finitely generated as $\mathbb{Q}[V_i]$-module, $A_i^1=0$, and $V_i\otimes A^0\rightarrow A^2$ is injective for $0\leq i \leq k$.
\item We have $(\id_{\mathbb{Q}[V_i]}\otimes_{\mathbb{Q}[V_j]} f_{jl})\circ f_{ij}=f_{il}$. Furthermore for any subspace $W\subset V$ and the maximal $V_j\in\{V_0,\ldots,V_k\}$ with the property that $V_j\subset W$, the map $f_{0j}$ becomes an isomorphism when localized at the multiplicative subset $S(W)\subset \mathbb{Q}[V]$ generated by $V\backslash W$.
\end{enumerate}
\end{thm}

\begin{proof}
Let $C$ be the collection of subgroups $H_0,\ldots,H_k$ where $H_i$ is the subtorus such that $\ker (H^*(BT)\rightarrow H^*(BH_i))=V_i$. Then $R_T\cong \mathbb{Q}[V]$ induces natural identifications $R_{T/H_i}\cong \mathbb{Q}[V_i]$. For later use we also note that $S_{T}(H_i)$ as in the definition of \textbf{(LC)} coincides with $S(V_i)$ as defined in $(iii)$ above.
Furthermore if $V_i+V_j=V_l$, then $H_i\cap H_j=H_l$. Thus $\mathcal{D}(C)$ consists of all tuples $(H_k,H_i)$ with $H_k\subset H_i$. Set $\underline{A}(H_k,H_i)=\mathbb{Q}[V_k]\otimes_{\mathbb{Q}[V_i]} A_i$ together with the trivial differential. Thus the $\underline{A}(H_k,H_i)$ and the maps induced by the $f_{ij}$ define a $\mathcal{D}(C)$-system $\underline{A}$ with an $\underline{R}$-structure. Clearly \textbf{(TC)} is fulfilled. If we can show that $\underline{A}$ satisfies \textbf{LC}, then the statement follows from Theorem \ref{thm:R-realization}. 

Let $(H_k,H_i)\in \mathcal{D}$ and $K\supset H_i$ a torus. Then $m_{\mathcal{D}(C)}(K)=H_j$ where $H_j$ is the minimal group in $C$ containing $K$. It corresponds to the maximal vector space $V_j$ contained in $W=\ker((H^*(BT)\rightarrow H^*(BK))$. We need to show that
\[\underline{A}(H_k,H_i)=\mathbb{Q}[V_k]\otimes_{\mathbb{Q}[V_i]} A_i\rightarrow \mathbb{Q}[V_k]\otimes_{\mathbb{Q}[V_j]} A_j=\underline{A}(H_k,H_j)\]
becomes an isomorphism after localizing at $S_{T/H_k}(K)$. We split $\mathbb{Q}[V]=\mathbb{Q}[V_k]\otimes \mathbb{Q}[V_k^\perp]$ for some decomposition $V=V_k\oplus V_k^\perp$. By assumption, in the commutative diagram
\[\xymatrix{
A_0\ar[r]^{f_{0i}}\ar[dr]_{f_{0j}} & \underline{A}(T,H_i)\ar[d]&\ar[l]_(.6)\cong \mathbb{Q}[V_k^\perp]\otimes \underline{A}(H_k,H_i)\ar[d]
\\ & \underline{A}(T,H_j) &\ar[l]_(.6)\cong \mathbb{Q}[V_k^\perp]\otimes \underline{A}(H_k,H_j)
}\]
$f_{0j}$ becomes an isomorphism when localised at $S_T(K)$ which is the set of all monomials generated in $V\backslash W$. Also $f_{0i}$ is an isomorphism when localized at $S_T(H_i)\subset S_T(K)$. Consequently the vertical maps become isomorphisms when localized at $S_T(K)$. Let $M$ denote the kernel of $\underline{A}(H_k,H_i)\rightarrow \underline{A}(H_k,H_j)$. Then we have shown that there is a monomial $f\in S_T(K)$ which annihilates $\mathbb{Q}[V_k^\perp]\otimes M$. We can write $f=\prod_i(x_i+y_i)$ with $x_i\in V_k$, $y_i\in V_k^\perp$. Let $f'$ be the product over all $x_i$ for which $y_i=0$. It follows from the bigrading on $\mathbb{Q}[V_k^{\perp}]\otimes M$ that it is also annihilated by $f'$. By assumption $x_i+y_i \in V\backslash W$ for all $i$ and in particular $x_i\in V_k\backslash W$ whenever $y_i=0$. Thus $f'\in S_{T/H_k}(K)$ and $S_{T/H_k}(K)^{-1}M=0$. For the cokernel one argues analogously, thus proving that $S_{T/H_k}(K)^{-1} \underline{A}(H_k,H_i)\rightarrow S_{T/H_k}(K)^{-1} \underline{A}(H_k,H_j)$ is an isomorphism.
\end{proof}

\begin{cor}\label{cor:discretealgebras}
Let $A$ be a $\mathbb{Q}[x]$-algebra which is finitely generated as a $\mathbb{Q}[x]$-module, spacelike, $A^1=0$, and $\langle x\rangle_\mathbb{Q}\otimes A^0\rightarrow A^2$ is injective. Assume further that for $S=\{x^k~|~k\geq 0\}$, the localized graded algebra $S^{-1}A$ has no nilpotent elements. Then $A$ is isomorphic as a $\mathbb{Q}[x]$-algebra to the equivariant cohomology of a finite $S^1$-CW-complex if and only if $(S^{-1}A)^0\cong \mathbb{Q}\times\ldots\times \mathbb{Q}$ as $\mathbb{Q}$-algebras. In this case the realization can be chosen with discrete fixed point set and $S^1$-simply connected. Conversely any equivariant cohomology algebra of such a space is of the above type.
\end{cor}

\begin{proof}
For such an $A$ we set $A_1=(S^{-1}A)^0\cong\mathbb{Q}\times\ldots\times \mathbb{Q}$ and $V_1=0$. Note that the map $f_{01}\colon A\rightarrow (S^{-1} A)^{\geq 0}=\mathbb{Q}[x]\otimes A_1$ induces an isomorphism when localized at $S$. Then by Theorem \ref{thm:cohorealization} there is a finite $S^1$-simply connected $S^1$-CW-complex which realizes $A$. Clearly the realization can be chosen with a finite number of fixed points.

Conversely assume $X$ is a finite $S^1$-CW-complex such that $S^{-1}H_{S^1}^*(X)$ contains no nilpotent elements. By Borel localization we have $S^{-1}H_{S^1}^*(X)\cong S^{-1}H_{S^1}^*(X^{S^1})\cong \mathbb{Q}[x,x^{-1}]\otimes H^*(X^{S^1})$. Thus the nonequivariant cohomology $H^*(X^{S^1})$ must not contain any nilpotent elements. This means that it is isomorphic to $\mathbb{Q}\times\ldots\times\mathbb{Q}$. Furthermore the equivariant cohomology of any finite $S^1$-simply connected $S^1$-CW-complex with discrete fixed points fulfils these algebraic requirements.
\end{proof}

\begin{ex}\label{ex:nonrealizable}
Consider the graded $\mathbb{Q}[x]$-algebra $A_c:=\mathbb{Q}[x,a]/(a^2-cx^2)$ for some scalar $c\in \mathbb{Q}$ and $x,a$ of degree $2$. Then $(S^{-1}A_c)^0=\langle 1, \frac{a}{x}\rangle_\mathbb{Q}$ with the $\mathbb{Q}$-algebra structure defined by the relation $\left(\frac{a}{x}\right)^2=c\cdot 1$. If $c$ is not a square, then this is isomorphic to $\mathbb{Q}[\sqrt{c}]$ and it follows that $A_c$ is not realizable by a finite $S^1$-CW-complex. Note that it is still realizable by an infinite $S^1$-CW-complex. In case $c$ is a square, sending the $\mathbb{Q}$-basis $\frac{1}{2}+\frac{1}{2\sqrt{c}}\frac{a}{x}$, $\frac{1}{2}-\frac{1}{2\sqrt{c}}\frac{a}{x}$ to $(1,0),(0,1)\in\mathbb{Q}\times\mathbb{Q}$ gives a $\mathbb{Q}$-algebra isomorphism. Thus $A_c$ is realizable by a finite $S^1$-CW-complex. In fact $A_c$ can be checked to be isomorphic to the equivariant cohomology algebra of the standard $S^1$-action on $S^2$. We observe that while e.g.\ $A_1$, $A_2$ behave differently with respect to finite realization, we have $A_1\otimes_\mathbb{Q}\mathbb{R}\cong A_2\otimes_\mathbb{Q}\mathbb{R}$ as $(R\otimes_\mathbb{Q}\mathbb{R})$-algebras. Thus finite realization relies on the choice of coefficient field. 
\end{ex}

\subsection{Realization of graph cohomology} \label{sec:GKM}
A GKM manifold is a smooth, compact, orientable manifold $M$ with $H^{odd}(M)=0$ together with a smooth action of a torus $T$ with discrete fixed points, such that the one skeleton $M_1=\{x\in M~|~ \dim (T\cdot x)\leq 1\}$ is a finite union of copies of $S^2$ (some sources also require the manifold to be almost complex). These manifolds are named after the authors of \cite{GKM} and their main appeal comes from the fact that $M_1$ can be encoded in a labelled graph, from which $H_T^*(M)$ can be computed via the purely combinatorial notion of graph cohomology. On the combinatorial side, a notion of an abstract GKM graph has been defined in \cite{GuilleminZara}. The definition, which we will not review in detail, encompasses several necessary conditions for a labelled graph to be realizable by a GKM manifold. However, currently there is no general answer to the question whether a given abstract GKM graph is in fact realizable by a GKM manifold. The aim of this section is to solve the realization problem in absence of the manifold condition while appropriately relaxing the conditions on the combinatorial side.

For a graph $\Gamma$ we denote by $E(\Gamma)$ the set of edges and by $V(\Gamma)$ the set of vertices. We consider edges to be nonoriented. Motivated by the notion of GKM graphs, we work with the following more general objects.

\begin{defn}
A $T$-graph is a finite graph $\Gamma$ together with a labelling function $\alpha\colon E(\Gamma)\rightarrow H^2(BT;\mathbb{Z})\backslash\{0\}$. Considering $H^*(BT;\mathbb{Z})\subset R_T$, we define the (rational) graph cohomology to be the $R_T$-algebra defined as
\[H^*(\Gamma)=\{x\in R_T^{V(\Gamma)}~|~ x(p)-x(q)\in (\alpha(e)) \text{ for any }e\in E(\Gamma)\text{ connecting }p,q\in V(\Gamma)\}.\]
\end{defn}

Any $T$-graph $\Gamma$ defines a $T$-space as follows: take a copy of $S^2$ for every edge $e\in E(\Gamma)$ and endow it with the following $T$-action: up to sign, there is a unique homomorphism $\varphi_e\colon T\rightarrow S^1$ such that $\alpha(e)$ generates the image of $H^2(BS^1;\mathbb{Z})\rightarrow H^2(BT;\mathbb{Z})$ and we pull back the standard rotation $S^1\curvearrowright S^2$ along $\varphi_e$.
Note that the equivariant homeomorphism type of the result does not depend on the sign choice. This action on $S^2$ has two fixed points (the two poles) which we identify with the two vertices associated to the edge $e$. Now glue the copies of $S^2$ together at the fixed points in the way prescribed by $\Gamma$. The resulting space is denoted by $R(\Gamma)$. We note that for GKM manifold $M$, we have $M_1= R(\Gamma)$ for a $T$-graph $\Gamma$. It is unique up to the signs of the labels.

\begin{lem}\label{lem:tree}
Let $\Gamma$ be a $T$-graph which is a tree. Then $H^*_T(R(\Gamma))\rightarrow H^*_T(R(\Gamma)^T)$ is injective.
\end{lem}

\begin{proof}
Since $\Gamma$ is a tree it follows that $R(\Gamma)$ is (non-equivariantly) homotopy equivalent to a wedge of two-spheres. In particular cohomology is concentrated in even degrees. Thus the Serre spectral sequence of the Borel fibration collapses at the $E_2$ page and consequently $H^*_T(R(\Gamma))$ is a free $R_T$-module. Since the kernel of $H^*_T(R(\Gamma))\rightarrow H^*_T(R(\Gamma)^T)$ is $R_T$-torsion by Borel localization, it needs to vanish.
\end{proof}

For a $T$-graph $(\Gamma,\alpha)$ and a subgroup $H\subset T$, we denote by $\Gamma^H$ the subgraph with the same vertex set and whose edge set consists of those edges $e\in E(\Gamma)$ such that $\alpha(e)$ lies in the kernel of $H^2(BT;\mathbb{Z})\rightarrow H^2(BH;\mathbb{Z})$.
\begin{thm}\label{thm:GKM-realization}
Let $(\Gamma,\alpha)$ be a $T$-graph. Then there is a finite $T$-CW-complex $X$ with $H^*_T(X)=H^*(\Gamma)$ and whose one-skeleton $X_1=\{x\in X~|~\mathrm{codim} T_x\leq 1\}$ is equivariantly homeomorphic to $R(\Gamma)$ if and only if for any codimension $1$ subtorus $H\subset T$, the graph $\Gamma^H$ is a disjoint union of trees. In this case $X$ can be chosen to be $T$-simply connected.
\end{thm}

\begin{proof}
For a given $T$-graph $\Gamma$, let $X$ be a finite $T$-CW-complex with $X_1=R(\Gamma)$. Assume that for some subtorus $H\subset T$ the graph $\Gamma^H$ is not a tree. Since $X^H=R(\Gamma^H)$, it follows that $H_T^1(X^H)\neq 0$. By \textbf{(TC)} we deduce that $S_T(H)^{-1}H_T^*(X^H)$ does not vanish in odd degrees. Borel localization \textbf{(LC)} implies that $H^{odd}_T(X)\neq 0$. Thus we can not have $H_T^*(X)\cong H^*(\Gamma)$.

For the converse direction, let $C$ consist of all isotropy groups of $R(\Gamma)$ and set $\mathcal{D}:=\mathcal{D}(C)\subset \mathcal{S}$. The $\mathcal{D}$-system $H^*(\underline{A_{PL}(R(\Gamma))})$ satisfies \textbf{(TC)}, \textbf{(LC)}, and $H^*(R(\Gamma)^H_{T/U})$ is finitely generated as an $R_{T/U}$-module for all $(U,H)\in \mathcal{D}$. But then clearly the same holds for the $\mathcal{D}$-system $\underline{A}:=H^{even}(\underline{A_{PL}(R(\Gamma))})$ where we have set the odd-dimensional degrees equal to $0$. A Mayer-Vietoris argument shows that indeed $H^*(\Gamma)=H_T^{even}(R(\Gamma))$. Thus at this point realizability of $H^*(\Gamma)$ follows from Theorem \ref{thm:R-realization}. Following the realization procedure, one can check that the realization can be chosen with $X_1=R(\Gamma)$. A slightly less direct but maybe faster way to see this is the following. Let $\underline{P}\rightarrow \underline{A_{PL}(*)}$ be as in Theorem \ref{thm:P-realization} and consider the subset $\mathcal{C}\subset \mathcal{D}$ of those $(U,H)\in \mathcal{D}$ with $H$ of codimension $1$ or less. Furthermore let $\underline{A}\rightarrow \underline{A'}$ be a quasi isomorphism such that $\underline{A'}$ satisfies \textbf{(SC)} and fix a strictification $\underline{P}\rightarrow \underline{A'}$ of the cohomology $\underline{R}$-structure on $\underline{A}$ as in Corollary \ref{cor:strictreplacement}.

We claim that $\underline{A'}|_\mathcal{C}$ is connected to $\underline{A_{PL}(R(\Gamma))}|_\mathcal{C}$ by a chain of quasi isomorphisms of systems with $\underline{P}$-structures. If this holds, then the inductive realization procedure in the proof of Theorem \ref{thm:P-realization} can be continued by gluing cells of isotropy codimension $\geq 2$ to $R(\Gamma)$, extending the quasi isomorphism of $\mathcal{C}$-systems to one of $\mathcal{D}$-systems. This yields the desired realization.

Let $Y$ be any $T$-space such that $\underline{A_{PL}(Y)}$ is quasi isomorphic to $\underline{A}$ (as provided by Theorem \ref{thm:P-realization}). Via the quasi isomorphism, the vertices of $R(\Gamma)$ correspond to the path components of $Y^T$ and we choose an image of each vertex in the respective component. Any $S^2$ in $R(\Gamma)$ with principal isotropy $H$ can be seen as a $1$-cell $[0,1]\times T/H$ glued to the respective vertices. The existence of the quasi isomorphism implies that two vertices lie in the same component of $R(\Gamma)^H$ if and only if their images lie in the same path component of $Y^H$. Thus we can extend the map over the $1$-cells to obtain an equivariant map $R(\Gamma)\rightarrow Y$. This induces for each $(U,H)\in \mathcal{C}$ a commutative diagram
\[\xymatrix{
H^*_{T/U}(R(\Gamma)^H)\ar[d] & H^*_{T/U}(Y^H)\ar[d]\ar[l]\\
H^*_{T/U}(R(\Gamma)^T) & H^*_{T/U}(Y^T)\ar[l]
}\]
in which the bottom horizontal map is the identity when identifying $H^*_{T/U}(Y^T)\cong \underline{A}(U,T)= H^{even}_{T/U}(R(\Gamma)^T)$ via the given quasi isomorphisms. The images of the vertical maps agree since $H^*_{T/U}(R(\Gamma)^T)$ is concentrated in even degrees. By Lemma \ref{lem:tree}, the vertical maps are injective and thus the top horizontal map is an isomorphism. This proves the claim.
\end{proof}

\begin{rem}\label{ex:Egraph} For a $T$-graph to be an abstract GKM graph in the sense of \cite{GuilleminZara}, it is in particular required that two adjacent edges have linearly independent labels. In this case, the subgraphs $\Gamma^H$ in the requirements of Theorem \ref{thm:GKM-realization} are disjoint unions of single vertices and single edge graphs. Thus the theorem is applicable and all abstract GKM graphs are realizable by finite $T$-CW-complexes. 

Recall that in the definition of a GKM manifold $M$ we require $H^{odd}(M)=0$. With respect to this condition our realizations behave as follows:
under the assumption that the fixed point set of a $T$-space $X$ is discrete, the condition $H^{odd}(X)=0$ is equivalent to the fact that $H^*_T(X)$ is free as an $R_T$-module (observe that the Serre spectral Sequence of the Borel fibration degenerates at $E_2$ if and only if $H^{odd}(X)=0$; then use \cite[Corollary 4.2.3]{AP}). Consequently, any of our realizations from Theorem \ref{thm:GKM-realization} will fulfil the condition of vanishing odd cohomology if and only if the graph allows such a realization, i.e.\ if the graph cohomology is free over $R_T$.

\end{rem}

\section{The equivariant rational homotopy category}\label{sec:homcats}

An equivariant rational equivalence is an equivariant map $X\rightarrow Y$ between $T$-spaces such that the induced maps $X^H\rightarrow Y^H$ are rational equivalences for all $H\subset T$. The equivariant rational homotopy category is the localization of the category of $T$-spaces at the equivariant rational equivalences. The goal here is to find an algebraization of this category.

For some $\mathcal{D}\subset \mathcal{S}$, denote by $\mathcal{A}^{\mathcal{D}}$ the category of of $\mathcal{D}$-systems $\mathcal{D}\rightarrow \mathrm{cdga}^{\geq 0}$. For some $\underline{P}$ in $\mathcal{A}^{\mathcal{D}}$ let $\mathcal{A}^{\mathcal{D}}_{\underline{P}}$ denote the category of $\mathcal{D}$-systems under $\underline{P}$, i.e.\ objects are morphisms $\underline{P}\rightarrow \underline{A}$ and morphisms are morphisms of $\mathcal{D}$-systems which are compatible with the $\underline{P}$-structures. The following algebraization was achieved in \cite[Theorem 4.4]{ScullMendes}.

\begin{thm}\label{thm:scullmendes}
Let $\underline{P}\rightarrow A_{PL}(*)$ be a quasi isomorphism where $\underline{P}$ is free nilpotent. Then the functor $X\mapsto \underline{A_{PL}(X)}$ induces an embedding of rational homotopy category of $T$-simply connected $T$-finite $\mathbb{Q}$-type spaces as a full subcategory of the homotopy category of $\mathcal{A}_{\underline{P}}^{\mathcal{S}}$.
\end{thm}

We have adjusted the formulation slightly and restricted the result to tori and $\mathbb{Q}$-coefficients to fit the presentation of this paper. We note that a description of the image of this embedding (similar to the corresponding parts of Theorem \ref{thm:P-realization}) is given in \cite[Theorem 3.7]{ScullMendes}. This results in an algebraic description of the $T$-equivariant rational homotopy category. As the authors of \cite{ScullMendes} state in their introduction, a drawback of the algebraization above lies in the fact that the object $\underline{P}$ is in general hard to describe explicitly. We opt to simplify the algebraization by showing that instead of a specific object $\underline{P}$, we only need to keep track of cohomological data in the form of cohomology $\underline{R}$-structures. Let $\widetilde{\mathcal{A}}^{\mathcal{D}}_{\underline{R}}$ denote the category whose objects are $\mathcal{D}$-systems $\mathcal{D}\rightarrow\mathrm{cdga}^{\geq 0}$ with cohomology $\underline{R}$-structures and whose morphisms are morphisms of systems which respect the cohomology $\underline{R}$-structures.

\begin{rem}\label{rem:homcats}
We wish to form the homotopy categories of $\mathrm{Ho}(\mathcal{A}_{\underline{P}}^\mathcal{D})$ and $\mathrm{Ho}(\widetilde{\mathcal{A}}^\mathcal{D}_{\underline{R}})$ by localizing with respect to quasi isomorphisms. To do this, we fix the following model structures. On the category of non-negatively graded, unital cdgas we choose the model structure introduced in \cite{BG}, whose weak equivalences are quasi isomorphisms and whose fibrations are surjective maps. Then by \cite[Theorem 11.6.1]{Hirschhorn} we have an induced model structure on $\mathcal{A}^\mathcal{D}$ whose fibrations and weak equivalences are just defined as pointwise surjections and quasi isomorphisms. The undercategory $\mathcal{A}^\mathcal{D}_{\underline{P}}$ inherits this model structure by defining morphisms to be fibrations and weak equivalences if the forgetful morphisms in $\mathcal{A}^\mathcal{D}$ are (\cite[Theorem 7.6.5]{Hirschhorn}).

This guarantees the existence of $\mathrm{Ho}(\mathcal{A}_{\underline{P}}^\mathcal{D})$. The category $\widetilde{\mathcal{A}}^\mathcal{D}_{\underline{R}}$ is not an undercategory of $\mathcal{A}^\mathcal{D}$. In fact, it is not technically a model category, e.g.\ due to the lack of an initial object. Still, the localization $\mathrm{Ho}(\widetilde{\mathcal{A}}^\mathcal{D}_{\underline{R}})$ exists and can be explicitly constructed as follows: let $C,F$ denote cofibrant and fibrant replacement functors on $\mathcal{A}^\mathcal{D}$ and let $Q\colon \widetilde{\mathcal{A}}^\mathcal{D}_{\underline{R}}\rightarrow \mathcal{A}^\mathcal{D}$ denote the forgetful functor. Then the objects of $\mathrm{Ho}(\widetilde{\mathcal{A}}^\mathcal{D}_{\underline{R}})$ are defined to be the same as those of $\widetilde{\mathcal{A}}^\mathcal{D}_{\underline{R}}$. Morphisms $A\rightarrow B$ are given by the homotopy classes in $\mathcal{A}^\mathcal{D}$ of those morphisms in $\hom(FCQ(A),FCQ(B))$ which respect the cohomology $\underline{R}$-structures inherited from $A$ and $B$.

To see that this is indeed a localization at the quasi isomorphisms, we make the following observation: as $CQ(A)\simeq Q(A)$ inherits a cohomology $\underline{R}$-structure from $A$, we obtain an endofunctor $\tilde{C}$ on $\widetilde{\mathcal{A}}^\mathcal{D}_{\underline{R}}$ with $Q\tilde{C}=CQ$ and a natural quasi isomorphism $\tilde{C}A\rightarrow A$ in $\widetilde{\mathcal{A}}^\mathcal{D}_{\underline{R}}$. Analogously, $F$ lifts to a functor $\tilde{F}$.
Now the claim can be proved analogously to \cite[Theorem 8.3.5]{Hirschhorn} while replacing the roles of $C,F$ with $\tilde{C},\tilde{F}$ in the appropriate places.

\end{rem}

The following theorem is in essence a translation of Theorem \ref{thm:strictification}. In the theorem below, parts $(i)$, $(ii)$, and $(iii)$ of \ref{thm:strictification} correspond to essential surjectivity, fullness, and faithfulness respectively.

\begin{thm}\label{thm:categorystricitification}
Let $\mathcal{D}\subset \mathcal{S}$ be a finite stable subset. Let $\underline{P}\rightarrow \underline{A_{PL}(*)}$ be a weak quasi isomorphism where $P(U,H)$ is a free nilpotent cdga and $P(U,H)=P(U,H')$ for all $U\subset H\subset H'$. Then the functor $\mathrm{Ho}(\mathcal{A}^{\mathcal{D}}_{\underline{P}})\rightarrow \mathrm{Ho}(\widetilde{\mathcal{A}}^\mathcal{D}_{\underline{R}})$ induces an equivalence between the full subcategories on the spacelike and cohomologically simply connected objects satisfying \textbf{(TC)}.
\end{thm}

\begin{proof}
By Corollary \ref{cor:strictreplacement}, for any spacelike object $\underline{A}$ of $\widetilde{\mathcal{A}}^\mathcal{D}_{\underline{R}}$ with $H^1(\underline{A})=0$ and \textbf{(TC)}, we have a quasi isomorphism $\underline{A}\rightarrow \widehat{\underline{A}}$ such that $\widehat{\underline{A}}$ satisfies \textbf{(SC)} and a morphism $\underline{P}\rightarrow \widehat{\underline{A}}$ which induces the cohomology $\underline{R}$-structure induced by $A$. In particular $\widehat{\underline{A}}$ is an object of ${\mathcal{A}}^\mathcal{D}_{\underline{P}}$ whose image in $\mathrm{Ho}(\widetilde{\mathcal{A}}^\mathcal{D}_{\underline{R}})$ is isomorphic to that of $\underline{A}$. Thus we have shown essential surjectivity.

We fix such a preimage for every object. If $\underline{B}$ is another such object of $\widetilde{\mathcal{A}}^\mathcal{D}_{\underline{R}}$ and $\underline{A}\rightarrow \underline{B}$ is a morphism in $\hom_{\widetilde{\mathcal{A}}^\mathcal{D}_{\underline{R}}}(\underline{A},\underline{B})$, then as in \ref{prop:SCexistence} we can extend to a map of systems $\widehat{\underline{A}}\rightarrow\widehat{\underline{B}}$. This map is not necessarily compatible with the fixed chosen $\underline{P}$-actions but it is so up to pointwise homotopy. Let $\widehat{\underline{B}}'$ be the same as $\widehat{\underline{B}}$ but with the $\underline{P}$-action induced by $\underline{\widehat{A}}\rightarrow \underline{\widehat{B}}$. Then
by part $(ii)$ of Theorem \ref{thm:strictification} we obtain a homotopy $\underline{P}\rightarrow \underline{\widehat{B}\otimes (t,dt)}$ such that in the commutative diagram
\[\xymatrix{
\widehat{\underline{A}}\ar[r] & \widehat{\underline{B}}' & \underline{\widehat{B}\otimes (t,dt)}\ar[l]\ar[r] & \widehat{\underline{B}}\\
{\underline{A}}\ar[r]\ar[u] & {\underline{B}}\ar[u] & \underline{{B}\otimes (t,dt)}\ar[u]\ar[l]\ar[r] & {\underline{B}}\ar[u]
}\]
all morphisms in the top row are compatible with the fixed $\underline{P}$-actions. Also note that if we already have a (strict) $\underline{P}$-action on $\underline{A}$ or $\underline{B}$ then the $\underline{P}$-actions on $\widehat{\underline{A}}$ and $\widehat{\underline{B}}$ can be chosen such that the respective outer vertical map is compatible with the $\underline{P}$-actions as well. Furthermore the string of maps  from $\underline{A}$ to $\underline{B}$ which goes through the top row and the outer vertical maps defines a morphism in $\mathrm{ho}(\widetilde{\mathcal{A}}^\mathcal{D}_{\underline{R}})$ which is equivalent to the original morphism $\underline{A}\rightarrow \underline{B}$.

Now for $\underline{A}$ and $\underline{B}$ as above with a fixed $\underline{P}$-action, any morphism in $\mathrm{Ho}(\widetilde{\mathcal{A}}^\mathcal{D}_{\underline{R}})$ between them is represented by a string of morphisms in $\widetilde{\mathcal{A}}^\mathcal{D}_{\underline{R}}$ through objects which are spacelike, cohomologically simply connected and satisfy \textbf{(TC)} (cf. Remark \ref{rem:homcats}). By what we have seen above, this string is equivalent in $\mathrm{Ho}(\widetilde{\mathcal{A}}^\mathcal{D}_{\underline{R}})$ to a string of morphisms which comes from ${\mathcal{A}}^\mathcal{D}_{\underline{P}}$. This shows fullness.

It remains to show that the functor between the homotopy categories is faithful. It suffices to show that this holds for the forgetful functor $\mathrm{Ho}(\mathcal{A}^\mathcal{D}_{\underline{P}})\rightarrow \mathrm{Ho}( \mathcal{A}^\mathcal{D})$.

Consider two morphisms $f,g\colon \underline{A}\rightarrow \underline{B}$ between fibrant and cofibrant objects $\underline{A},\underline{B}$ of $\mathcal{A}^\mathcal{D}_{\underline{P}}$ which map to the same morphism in $\mathrm{Ho}(\mathcal{A}^\mathcal{D})$. We need to show that they also get identified in $\mathrm{Ho}(\mathcal{A}^\mathcal{D}_{\underline{P}})$. Let $\underline{\mathcal{M}}\rightarrow \underline{A}$ be a quasi isomorphism in $\mathcal{A}^\mathcal{D}_{\underline{P}}$ where $\underline{\mathcal{M}}$ is a free disconnected extension of $\underline{P}\rightarrow \underline{A}$ (cf.\ Proposition \ref{prop:Sullivandiagram}). With the notation from above, Proposition \ref{prop:SCexistence} provides a commutative diagram
\[\xymatrix{
\underline{\mathcal{M}}\ar[r]\ar[d] & \underline{A}\ar[r]^f & \underline{B}\ar[d]\ar[d]\\
\widehat{\underline{\mathcal{M}}}\ar[rr]^{\hat{f}} & &\widehat{\underline{B}}
}\]
where the vertical maps are quasi isomorphisms. Analogously we obtain a map $\hat{g}$ and it suffices to prove that $\hat{f}$ and $\hat{g}$ are the same in $\mathrm{Ho}(\mathcal{A}^\mathcal{D}_{\underline{P}})$. We show that there is a homotopy $\underline{\widehat{\mathcal{M}}}\rightarrow \underline{\widehat{B}\otimes (t,dt)}$ between $\hat{f}$ and $\hat{g}$. In this case by part $(iii)$ of Theorem \ref{thm:strictification} there is also a homotopy relative to $\underline{P}$ which finishes the proof.

To prove this, consider cofibrant replacements $\underline{C\mathcal{M}}\rightarrow \underline{\widehat{\mathcal{M}}}$ and $\underline{C\widehat{B}}\rightarrow \widehat{\underline{B}}$ in $\mathcal{A}^\mathcal{D}$. We note that these are also fibrant since all objects are fibrant. Then $\hat{f}$, $\hat{g}$ lift to morphisms $\tilde{f},\tilde{g}\colon  \underline{C\mathcal{M}}\rightarrow \underline{C\widehat{B}}$ which are homotopic in the sense of the model structure on $\mathcal{A}^\mathcal{D}$. We may use $\underline{C\widehat{B}\otimes (t,dt)}$ as a path object. Again using Proposition \ref{prop:SCexistence} we obtain the commutative diagram of solid arrows\vspace{0.3cm}

\[\xymatrix{
\widehat{\underline{C\mathcal{M}}}\ar[r]_p\ar@/^2.0pc/@{-->}[rrr]^\phi& \widehat{\underline{\mathcal{M}}} \ar[r]^{\hat{f}\times\hat{g}}& \widehat{\underline{B}}\times\widehat{\underline{B}}& \ar[l] \underline{\widehat{B}\otimes (t,dt)}\\
 & \underline{C\mathcal{M}}\ar[lu]\ar[u] \ar[r]^{\tilde{f}\times\tilde{g}} \ar@/_2.0pc/[rr] &\underline{C\widehat{B}}\times \underline{C\widehat{B}}\ar[u]& \ar[l] \underline{C\widehat{B}\otimes (t,dt)}\ar[u]
}\]\vspace{0.4cm}

\noindent in which $p$ satisfies \textbf{(SCM)}. Applying \ref{prop:SCexistence} to $\underline{C\mathcal{M}}\rightarrow \underline{\widehat{B}\otimes (t,dt)}$, we extend the homotopy to yield the dashed arrow $\phi$. The extension may be constructed  such that it is compatible with the already exiting map $\widehat{\underline{C\mathcal{M}}}\rightarrow \underline{\widehat{B}}\times\underline{\widehat{B}}$, by using the fact that for any $(U,H)\in \mathcal{D}$ the morphism from $\underline{\widehat{B}}(U,H)\otimes (t,dt)$ into the pullback of \[\underline{\widehat{B}}(U,H)\times\underline{\widehat{B}}(U,H)\rightarrow \lim_{\mathcal{D}(U,H)} \widehat{\underline{B}}\times\widehat{\underline{B}}\leftarrow \lim_{\mathcal{D}(U,H)}\widehat{\underline{B}}\otimes (t,dt)\] is a surjection. Thus the whole diagram commutes. It follows from Proposition \ref{prop:lifting} that $p$ admits a section $s$. The desired homotopy is given by $\phi\circ s$.
\end{proof}

\begin{lem}\label{lem:restrictiongalore}
Let $\mathcal{D}\subset \mathcal{S}$ be a stable subset and define $\mathcal{D'}$ as the stable subset $\{(U,H)\in \mathcal{S}~|~ H\in \mathcal{D}_R\}$. With $\underline{P}$ as above, let $\mathcal{B}_0$ be the full subcategory of $\mathrm{Ho}(\mathcal{A}^\mathcal{S}_{\underline{P}})$ on objects $\underline{A}$ which satisfy \textbf{(TC)} and for which $i_{\mathcal{D}'}\circ r_{\mathcal{D}'}(\underline{A})\rightarrow \underline{A}$ is a quasi isomorphism.
Let $\mathcal{B}_1$ be the full subcategory of $\mathrm{Ho}(\mathcal{A}^\mathcal{D}_{\underline{P}})$ on objects $\underline{A}$ which satisfy \textbf{(TC)}. Then restriction defines an equivalence of categories $\mathcal{B}_0\rightarrow\mathcal{B}_1$.
\end{lem}

\begin{proof}
The functors $i_{\mathcal{D}'}$ and $r_{\mathcal{D}'}$ induce an equivalence of categories between $\mathcal{B}_0$ and the full subcategory $\mathcal{B}'$ of $\mathrm{Ho}(\mathcal{A}^{\mathcal{D}'}_{\underline{P}})$ on systems satisfying \textbf{(TC)}. We construct a functor $j\colon\mathcal{B}_1\rightarrow \mathcal{B}'$ which is inverse to the restriction $r_{\mathcal{D}}\colon \mathcal{B}'\rightarrow\mathcal{B}_1$.

The procedure from Proposition \ref{prop:Sullivandiagram}, which replaces a system under $\underline{P}$ with a quasi isomorphic disconnected free nilpotent extension of $\underline{P}$, is functorial. We denote the corresponding functor by $\Gamma_{\mathcal{D}'}\colon \mathcal{A}^{\mathcal{D}'}_{\underline{P}}\rightarrow \mathcal{A}^{\mathcal{D}'}_{\underline{P}}$. As the functor operates pointwise, we obtain a compatible functor $\Gamma_\mathcal{D}$ on $\mathcal{A}^{\mathcal{D}}_{\underline{P}}$ satisfying $r_{\mathcal{D}}\circ \Gamma_{\mathcal{D}'}=\Gamma_{\mathcal{D}}\circ r_\mathcal{D}$. For simplicity we just denote both funtors by $\Gamma$. Let $\underline{A}$ be a $\mathcal{D}$-system under $\underline{P}$.
For a subtorus $U\subset T$ we denote by $m^L_\mathcal{D}(U)$ the intersection of all tori in $\mathcal{D}_L$ which contain $U$. Now fix $(U,H)\in\mathcal{D}'$.
The algebra $\Gamma\underline{A}(m_\mathcal{D}^L(U),H)$ is of the form $\prod_i (\underline{P}(m_\mathcal{D}^L(U),H)\otimes \Lambda V_i,d)$.
We set \[\underline{j(A)}(U,H)= \underline{P}(U,H)\otimes_{\underline{P}(m^L_\mathcal{D}(U),H)} \Gamma \underline{A}(m^L_\mathcal{D}(U),H)= \prod_i (\underline{P}(U,H)\otimes \Lambda V_i,d).\]
For $(U,H)\leq (U',H')\in \mathcal{D}'$ we have $(m^L_\mathcal{D}(U),H)\leq (m^L_\mathcal{D}(U'),H')$ in $\mathcal{D}$ so $\underline{j(A)}$ naturally inherits the structure of a $\mathcal{D}'$-system under $\underline{P}$. It satisfies \textbf{(TC)} by Lemma \ref{lem:tensorTC}. We have $r_{\mathcal{D}}\circ j=\Gamma$ and a natural transformation $\Gamma\rightarrow \id_{\mathcal{A}^\mathcal{D}_{\underline{P}}}$ consisting of quasi isomorphisms.

If we start with a $\mathcal{D}'$-system $\underline{B}$, then
$(m^L_\mathcal{D}(U),H)\leq (U,H)$ induces a map \[\underline{P}(U,H)\otimes_{\underline{P}(m^L_\mathcal{D}(U),H)}\Gamma\underline{B}(m^L_\mathcal{D}(U),H)\rightarrow \Gamma\underline{B}(U,H).\]
We obtain natural transformations $j\circ r_{\mathcal{D}}\rightarrow \Gamma\rightarrow \id_{\mathcal{A}^{\mathcal{D}'}_{\underline{P}}}$. The second transformation is a quasi isomorphism and the first transformation is a quasi isomorphism due to \textbf{(TC)}. This shows that $j$ and $r_\mathcal{D}$ descend to equivalences of categories between $\mathcal{B}'$ and $\mathcal{B}_1$.
\end{proof}

We sum up the results of this section in the following Corollary. We point out that in $(ii)$,  we only require $\underline{P}\rightarrow \underline{A_{PL}(*)}$ to be a weak morphism. Thus we may choose e.g.\ $\underline{P}=\underline{R}$ to obtain a description via systems with strict $\underline{R}$-structures. 

\begin{cor}\label{cor:categorialdescription}
Let $C$ be a finite collection of subgroups and $\mathcal{D}=\mathcal{D}(C)$. The full subcategory on objects with isotropies in $C$ of the equivariant rational homotopy category of $T$-simply connected, $T$-finite $\mathbb{Q}$-type $T$-spaces is equivalent to
\begin{enumerate}[(i)]
\item the full subcategory of $\mathrm{Ho}(\widetilde{\mathcal{A}}^\mathcal{D}_{\underline{R}})$ on objects $\underline{A}$ which satisfy \textbf{(TC)} and for which $\underline{R}^2\otimes H^0(\underline{A})\rightarrow H^2(\underline{A})$ is injective, $H^1(\underline{A})=0$, and $H^*(\underline{A})$ is finite type and spacelike.
\item the full subcategory of $\mathrm{Ho}(\mathcal{A}^\mathcal{D}_{\underline{P}})$ (with $\underline{P}$ as in Theorem \ref{thm:categorystricitification}) on objects $\underline{A}$ which satisfy \textbf{(TC)} and for which $H^2(\underline{P})\otimes H^0(\underline{A})\rightarrow H^2(\underline{A})$ is injective, $H^1(\underline{A})=0$, and $H^*(\underline{A})$ is finite type and spacelike.
\end{enumerate}
\end{cor}

\begin{proof}
Let $\underline{P}$ be a free nilpotent system with a quasi isomorphism $\underline{P}\rightarrow \underline{A_{PL}(*)}$ such that $P(U,H)=P(U,H')$ for all $U\subset H\subset H'$. Also let the $\mathcal{A}_0$ be the full subcategory of $\mathrm{Ho}(\mathcal{A}^{\mathcal{S}}_{\underline{P}})$ on those objects $\underline{A}$ which satisfy \textbf{(TC)} and for which $H^2(\underline{P})\otimes H^0(\underline{A})\rightarrow H^2(\underline{A})$ is injective, $H^1(\underline{A})=0$, and $H^*(\underline{A})$ is finite type and spacelike.
By Theorem \ref{thm:scullmendes}, $\mathcal{T}\rightarrow \mathcal{A}_0$ is fully faithful, where $\mathcal{T}$ is the equivariant rational homotopy category $\mathcal{T}$ of $T$-simply connected finite $\mathbb{Q}$-type $T$-spaces.

By Lemma \ref{lem:isotropytypes}, restricting to the full subcategory $\mathcal{T}^C$ of $\mathcal{T}$ on spaces with isotropy groups in $C$, the image of the functor $X\mapsto \underline{A_{PL}(X)}$ lies in the full subcategory $\mathcal{A}_0^C\subset \mathcal{A}_0$ on those objects which also lie in $\mathcal{B}_0$ as defined in Lemma \ref{lem:restrictiongalore}.

Let $\mathcal{A}_1$ be the full subcategory of $\mathrm{Ho}(\mathcal{A}^\mathcal{D}_{\underline{P}})$ on those objects $\underline{A}$ which satisfy \textbf{(TC)} and for which $H^2(\underline{P})\otimes H^0(\underline{A})\rightarrow H^2(\underline{A})$ is injective, $H^1(\underline{A})=0$, and $H^*(\underline{A})$ is of finite type and spacelike. Then $r_\mathcal{D}\colon \mathcal{A}_0^C\rightarrow \mathcal{A}_1$ is fully faithful by Lemma \ref{lem:restrictiongalore}. It follows from Theorem \ref{thm:P-realization} that the composition $\mathcal{T}^C\rightarrow \mathcal{A}_0^C\rightarrow\mathcal{A}_1$ is essentially surjective.
But as a composition of two fully faithful functors this means it is an equivalence of categories. Then clearly the equivalence between the two full subcategories of $\mathrm{Ho}(\mathcal{A}^\mathcal{D}_{\underline{P}})$ and $\mathrm{Ho}(\widetilde{\mathcal{A}}^\mathcal{D}_{\underline{R}})$ from Theorem \ref{thm:categorystricitification} restricts to an equivalence between $\mathcal{A}_1$ and the desired category $\mathcal{A}_2$ described in $(i)$. For some free nilpotent system $\underline{P'}$ with a weak quasi isomorphism $\underline{P'}\rightarrow \underline{A_{PL}(*)}$, the equivalence of the category described in $(ii)$ and $\mathcal{A}_2$ follows from Theorem \ref{thm:categorystricitification} as well.
\end{proof}

\bibliography{Realizationbib}
\bibliographystyle{acm}

\end{document}